\documentclass[11pt]{article}
\RequirePackage[OT1]{fontenc}
\usepackage{amsthm,amsmath,amssymb,mathrsfs,graphicx}
\usepackage{verbatim,multirow,subfigure}
\RequirePackage[colorlinks,citecolor=blue,urlcolor=blue]{hyperref}

\RequirePackage[numbers]{natbib}

\newtheorem{proposition}{Proposition}[section]
\numberwithin{equation}{section}
\theoremstyle{plain}
\newtheorem{thm}{Theorem}[section]
\newtheorem{lemma}{Lemma}[section]
\newtheorem{remark}{Remark}[section]

\def\convinl2{\stackrel{\LL_2}{\longrightarrow}}

\def\LL{\mathbb L}

\def \loe{>_{L}}
\def \loeeq{\geq_{L}}

\newcommand{\mcL}{\mathcal{L}}

\newcommand{\lb}{\label}

\newcommand{\be}{\begin{equation}}
\newcommand{\ee}{\end{equation}}
\newcommand{\beqa}{\begin{eqnarray*}}
\newcommand{\eeqa}{\end{eqnarray*}}
\newcommand{\beqn}{\begin{eqnarray}}
\newcommand{\eeqn}{\end{eqnarray}}
\newcommand{\ba}{\begin{array}}
\newcommand{\ea}{\end{array}}
\newcommand{\bc}{\begin{center}}
\newcommand{\ec}{\end{center}}
\newcommand{\btab}{\begin{tabular}}
\newcommand{\etab}{\end{tabular}}

\newcommand{\plus}{\scalebox{1}{$+$}}

\begin{document}

\title{Estimation of Causal Invertible VARMA Models}

\author{Anindya Roy$^1$, Tucker S. McElroy$^2$ and Peter Linton$^1$\\
$^1$Department of Mathematics and Statistics\\\
University of Maryland Baltimore County,
Baltimore MD 21250\\
$^2$U.S. Census Bureau,
Suitland, MD 20746}
\date{}
\maketitle

\begin{abstract}
We present a re-parameterization of  vector autoregressive moving average  (VARMA) models that allows estimation of  parameters under the constraints of causality and invertibility. The parameter constraints associated with a causal invertible VARMA model are highly complex. Currently there are no procedures that can maintain the constraints in the estimated VARMA process, except in the special case of a vector autoregression (VAR), where some moment based causal estimators are available. Even in the VAR case, the available likelihood based estimators are not causal. The maximum likelihood estimator based on the full likelihood that does not condition on the initial observations  by definition satisfies the causal invertible constraints but optimization of the likelihood under the complex constraints is an intractable problem. The commonly used Bayesian procedure for VAR often has posterior mass outside the causal set because the priors are not constrained to the causal set of parameters. 

We provide  an exact mathematical solution to this problem. An $m$-variate VARMA$(p, q)$ process contains $(p+ q) m^2 + \binom{m+1}{2}$ parameters, which must be constrained to a subset of Euclidean space in order to guarantee causality and invertibility.  This space is implicitly described in this paper, through the device of parameterizing the entire space of  block Toeplitz matrices in terms of  positive definite matrices and  orthogonal matrices. The parameterization has connection to Schur- stability of polynomials and the associated Stein transformation that are often used in dynamical systems literature. As an important by-product of our investigation,  we generalize a classical result in dynamical systems to provide a characterization of Schur stable matrix polynomials. 

The stable parameterization has many advantages. Primarily, it allows a convenient way for obtaining likelihood based estimators under constraints. Such estimators can have positive impact on how forecasts are made based on VARMA models. 
Also, it will be an important tool for understanding the state of a linear dynamical system when the system is known to be stable. \\
\\
\noindent{\bf{keywords}}: Constrained estimation;Block Toeplitz matrix;Schur stability;reparameterization
\end{abstract}

\section{Introduction}

This article develops a method for estimating the parameters of a general vector autoregressive moving average (VARMA) model under the constraint that the estimated process is causal and invertible. To our knowledge there are no existing procedures or software that maintain the restrictions of causality and invertibility in the estimation process (except in the special case of vector autoregression where Yule-Walker estimators are guaranteed to be stable).   This paper fills that gap by producing a parameterization of the process that automatically maintains the constraints during estimation. The precise definition of causality and invertibility in the context of VARMA models is given later in Section~2.  Based on the fact that VARMA models are ubiquitous in time series applications and that  causality and invertibility are often imposed on the model (but not always enforced during the estimation process due to the complexity of the constraints) it seems important to be able to estimate VARMA processes under such constraints.

Vector autoregressive moving average models are the most common multivariate time series models \cite{lutkepohl} that find use in many fields, including macroeconomics, econometrics, dynamical systems and the  physical sciences. The subclass of vector autoregressive (VAR) models are popular models that have been advocated especially in macroeconomic literature \cite{sims1980} as appropriate choices for multiple time series models. While the univariate autoregressive moving average (ARMA) models are very popular due to their interpretability, the multivariate version provides enormous flexibility in terms of understanding the lag-dependence structures in multiple time series data;  they also provide better forecasts than univariate models in many cases.  In many applications of VARMA processes, especially in dynamical system estimation, the underlying process is assumed to be stable. Here we use stability as a general term, but more precise definitions in the context of VARMA models will be given later. Stability in the estimated process is often desired but sometimes not achieved due to the lack of procedures that guarantee stability of estimated VARMA processes.  Failing to impose stability in the estimated process can have adverse consequences, particularly when one is interested in understanding the long term behavior of the process, as is the case with long-horizon forecasting.

The subclass of autoregressive processes, due to their  linear lag structure,  allows parameter estimation via least squares, a great advantage over  more general autoregressive moving average (ARMA) models. Linearity makes it possible to develop estimators that operate under the assumption that the process is stable, an advantage shared by the multivariate version as  well. Within the frequentist paradigm, there are least squares based algorithms such as Yule-Walker, and Burg's algorithm that are guaranteed to estimate a stable  process. The unconstrained maximum likelihood estimator (MLE) obtained from the full stationary likelihood that accounts for the contribution of the initial observations is also guaranteed to be in the stable region.  It is known \cite{fuller} that likelihood based estimators generally have better finite sample performance than moment based estimators provided the assumed likelihood is approximately correctly specified and the full likelihood, that includes the contribution of the initial observations, is considered. Due to  nonlinearity introduced by the contribution of the initial observations to the likelihood equations, estimation using the full likelihood is considerably more challenging than those involving the conditional likelihood that condition on the initial observations. 
However,  the conditional likelihood is not  a stationary likelihood and hence estimators obtained based on the conditional likelihood, - either conditional maximum likelihood estimators which reduce to the ordinary least squares estimator for a Gaussian likelihood or Bayesian VAR estimators - fail to impose stability in the estimated process. While there are at least some procedures available in the VAR setting that allow stable estimation, there are none in the general VARMA case. The main reason is the extreme complexity of the VARMA likelihood and that of the constraints on the parameters under causality and invertibility. 

There is a long literature of likelihood computation and optimization for the VARMA model, including univariate ARMA models. Before modern computing power, researchers often developed various approximations for the likelihood under which approximate MLE were obtained and their properties were studied; \cite{whittle1951, durbin, tunnicliffe_wilson1973, ansley, nicholls1979, hannan_rissanen, godolphin, koreisha_pukkila, reinsel_etal1992, deFrutos, mauricio_1995, mauricio_1997, mauricio_2002, metaxoglou_smith}  provide an EM type algortihm along with a state-space formulation that makes likelihood computation considerably faster and improves convergence. However, none of the procedures guarantee that the estimated VARMA process is causal and invertible.  

Along with maximum likelihood estimation of VARMA models,  much attention has also been devoted toward Bayesian autoregression (BVAR) and prior specification for  such models. The popular choices for prior specification include those described in \cite{litterman, doan_etal,  kadiyala_karlsson1997, sims_zha1998}. These are generally priors on the coefficients of the regression equations in the VAR representation and they follow normal inverse chi-square/wishart type specifications. The supports of these prior distributions on the coefficients of the VAR models are generally the entire Euclidean space, and hence there is positive probability that the posterior estimates may lie outside the constraint set determined by causality and invertibility.  Depending on the sample size and the dimension of the parameter space, significant posterior probability may exist for estimating processes that are not stationary. This is highly undesirable in applications where long-term behavior of the underlying stationary system is being estimated. Recent interest in Bayesian macroeconomics has spawned several research articles in BVAR; see \cite{koop_korobilis}.  

 To avoid numerical complexities and instabilities of constrained optimization or constrained prior specification, one can re-parameterize the problem via transformation in a way such that the new parameters are unconstrained.  It will be convenient to designate the term `pre-parameters' to describe the re-parameterized parameters. The term `pre' is used to convey the fact that inference is obtained on the original parameters (or simply `parameters') based on the analysis of the pre-parameters. For example, in a Bayesian scheme the prior on the parameters are only those induced by the specified priors on the pre-parameters. Previous attempts of such pre-parameterization for univariate ARMA models include \cite{wise1956,barndorff_neilsen, marriott_smith, Quenneville_mcleod1992}. 

 It is the contention of the authors that in  complex constrained parameter problems, it is much easier to carry out inference under suitable parameter transformations.  Transformation of  parameters to quantities that are free or nearly free of constraints often allow feasible solutions to complicated constrained inference problems. There are numerous instances of such approaches in statistics. Parameter transformation has been successfully used in many complicated constrained  estimation problems, such as estimation of covariance matrices under positive-definiteness constraints \cite{lindstrom_bates, leonard_hsu, pinheiro_bates1996}, order-constrained parameters \cite{dunsonneelon} and estimation of parameters under general polyhedral constraints \cite{danaher_etal}. In this article, we provide a parameterization of the VARMA process that naturally maintains the stability constraints. More generally,  we would like to advocate the general principle of parameter transformation as a  means of making inference under complicated parameter constraints.

\section{Causal Invertible VARMA Process}
In this section we review the VARMA models and motivate the need for reparameterization in estimation of stable VARMA models. 
\subsection{Background}
 Consider a mean zero stationary VARMA$(p, q)$  process of dimension $m$, denoted by $\{ X_t \}$, where the autocovariance matrices are given by $\Gamma (k) = \mathbb{E} [ X_{t} X_{t-k}^{\prime} ]$ for $k \geq 0$.  Let the autoregressive and moving average coefficient matrices be denoted by $\Phi_1, \cdots, \Phi_p$ and $\Theta_1, \ldots, \Theta_q,$ respectively.  Thus, we have the process defined by  the relation
\be
 \Phi(B) X_t = \Theta(B) Z_t
\lb{varma}
\ee
 for each $t$, where the innovations, $\{ Z_t \},$ are uncorrelated mean zero variables with covariance matrix $\Sigma$, $B$ is the back-shift operator and for any complex number $z \in \mathbb{C}$, $\Phi(z)$ and $\Theta(z)$ are polynomials defined as 
\beqn
\Phi(z) &=& 1_m - \Phi_1 z - \cdots - \Phi_p z^p, \nonumber \\
\Theta(z) &=& 1_m + \Theta_1 z + \cdots + \Theta_q z^q,
\lb{polynomials}
\eeqn
with $1_m$ denoting the identity matrix of dimension $m.$ The process is called a $causal$ VARMA process if \eqref{varma} has exactly one stationary solution of the form $X_t = \Psi(B) Z_t$ where $\Psi(z) = \sum_{j=0}^\infty \Psi_j z^j,\; z \in \mathbb{C}$  for a sequence of coefficient matrices $\{\Psi_j, j\geq 0\}$. Causality is a desirable property because it makes the process independent of future innovations, thereby making it possible to forecast ahead based on current and past observations. In terms of the polynomial $\Phi(z)$, the process is causal iff $det(\Phi(z)) \ne 0,$  for all $ z \in \mathbb{C}$ such that $|z| \leq 1$ (\cite{brockwell_davis} Theorem 11.3.1).  For what follows it will be convenient to characterize the causal process in terms of the associated polynomial ${\tilde{\Phi}}(z)$ defined as 
\[   {\tilde{\Phi}}(z) := z^{p}\Phi(z^{-1}) = z^p - \Phi_1 z^{p-1} - \cdots - \Phi_p.\]
Thus, a VARMA process defined by \eqref{varma} is causal iff $det({\tilde{\Phi}}(z)) \ne 0,$ for all $z \in \mathbb{C}$ such that $|z| \geq 1,$ i.e. all roots of ${\tilde{\Phi}}(\cdot)$ lie within the open unit disc $\mathcal{D} = \{z \in \mathbb{C}: |z| < 1\}. $  
Similarly, define 
\[{\tilde{\Theta}}(z) := z^q\Theta(z^{-1}) = z^q + \Theta_1 z^{q-1} + \cdots + \Theta_q.\]
  The VARMA$(p, q)$ is called {\it invertible} if, based on \eqref{varma}, the innovation process $Z_t$ can be given the representation $Z_t = \Pi(z) X_t$ in terms of a stationary solution $X_t,$ where $\Pi(z) = \sum_{j=0}^\infty \Pi_j z^j, z \in \mathbb{C}.$ Invertibility of the process is equivalent to the property that all roots of ${\tilde{\Theta}}(z)$ lie within the unit disc. 
A VARMA$(p, 0)$ process will be referred to as a VAR$(p)$ process (vector autoregression of order $p$) and a VARMA$(0, q)$ process will be referred to as a VMA$(q)$ process (vector moving average of order $q$). 

We will refer to ${\tilde{\Phi}}(z)$, ${\tilde{\Theta}}(z)$ and $\Sigma$ as the parameters of the VARMA$(p, q)$ process defined by \eqref{varma}. We will refer to ${\tilde{\Phi}}(z)$ as the autoregressive polynomial and to ${\tilde{\Theta}}(z)$ as the moving average polynomial associated with the VARMA representation \eqref{varma} and when it is clear from the context we will interchangeably refer to the associated coefficient matrix sequence $\Phi = (\Phi_1, \ldots, \Phi_p)$ and $\Theta = (\Theta_1, \ldots, \Theta_q)$ as the parameters as well. 
Before describing the parameter space of a causal invertible VARMA process, we introduce some notations.  Let $\loeeq$ denote the Loewner partial ordering for positive semi-definite matrices, i.e. for $m\times m$ matrices $A$ and $B$, $A \loeeq B$ means that $ A - B $ is a positive semi-definite matrix and $A \loeeq 0$ means that $A$ is positive semi-definite. The strict inequalities $A \loe B$ and $A \loe 0$ would mean that $A - B$ is positive definite and $A$ is positive definite, respectively.
Define
\be
{\mathscr{S}}^m_{++} =  \{ \Sigma \in {\mathscr{S}}^m : \Sigma \loe 0 \}
\ee
to be the set of all $m\times m$ symmetric positive definite matrices that constitute the interior of the convex cone, ${\mathscr{S}}^m_{+},$  of $m\times m$ positive semi-definite matrices in $ {\mathscr{S}}^m,$  the set of $m \times m$ symmetric matrices. 
A matrix polynomial $A(z) = z^k - A_1z^{k-1} - \cdots - A_k, $ will be called {\it Schur-Stable} if all roots of $A(z)$ lie within the unit disc $\mathcal{D}.$ Such polynomials are common in the dynamical systems literature \cite{bhatia,  kaszkurewicz_bhaya} and the study of such polynomials is closely associated with the stability of dynamical systems. Let 
\beqn
\mathfrak{S}^m_k &=& \{  A(z) = z^k - A_1z^{k-1} - \cdots - A_k: A_r \in \mathbb{R}^{m\times m}, \\
&& \qquad \qquad  r \geq 1,   \mbox {and } A(z) \mbox{ is Schur-Stable}\}  \nonumber 
\eeqn
define the set of all $m-$dimensional Schur-Stable matrix polynomials of degree $k$. 
 Let any polynomial $A(z) = z^k - A_1 z^{k-1} - \cdots - A_k$ be associated with the coefficient sequence $A = \left[ A_1, \ldots, A_k \right].$  Thus, when it is clear from the context we will define a sequence of matrices $A = \left[ A_1, \ldots, A_k\right]$ to be Schur-Stable provided the associated polynomial is Schur-Stable.

The innovation variance $\Sigma$ lies in ${\mathscr{S}}^m_{++}.$ The autoregressive polynomial ${\tilde{\Phi}}(z)$ and moving average polynomial ${\tilde{\Theta}}(z)$  for a causal  invertible VARMA$(p$, $ q)$ process will belong to $\mathfrak{S}^m_p$ and $\mathfrak{S}^m_q$, respectively. Thus, the parameters $(\Phi, \Theta, \Sigma)$ of an $m$-dimensional causal invertible VARMA$(p, q)$ process belongs to the parameter space given by
\be
\mathfrak{P} = \mathfrak{S}^m_p \times \mathfrak{S}^m_q \times {\mathscr{S}}^m_{++}.
\lb{parm}
\ee

\subsection{Estimation and Need for Reparameterization}

While modeling a multivariate time series using a (mean zero) VARMA$(p, q)$ process, one often makes distributional assumptions on the innovation process $Z_t$. Most often, the innovations are assumed to be Gaussian, i.e. $Z_t \sim N(0, \Sigma).$ Based on this assumption, a likelihood for the parameters $(\Phi, \Theta, \Sigma)$ can be written down  and used for likelihood based inference. Under the assumption of stationarity, for any $p \geq 0,$ define
\be
\Gamma_p = \begin{pmatrix}
\Gamma(0) & \Gamma(1) & \cdots & \Gamma(p) \\
\Gamma(-1)& \Gamma(0) & \cdots & \cdot              \\
\vdots        &  \ddots        & \ddots & \vdots             \\
\Gamma(-p) & \cdots     & \Gamma(-1) & \Gamma(0)
\end{pmatrix}
\lb{gammap}
\ee
to be the covariance matrix of $ \mbox{Vec} [ X_{t}, X_{t-1}, \cdots, X_{t-p}]$ where $Vec$ is the operation of stacking the column vectors following the $Vec$.  Thus, the  $jk$th block of $\Gamma_p$ is an $m\times m$ matrix, given by $\Gamma (k-j) = \Gamma^{\prime} (j-k)$, for $1 \leq j,k \leq (p+1)$. Assuming that the innovations are normally distributed, a stationary likelihood for $(\Phi, \Theta, \Sigma)$ based on a sample $X^{\prime} = (X_n, \ldots, X_1)$ (written in the reverse order for notational consistency) is 
\be
\mathcal{L}(\Phi, \Theta, \Sigma) = (2\pi)^{-n/2}[det(\Gamma_{n-1})]^{-1/2} \exp{(-0.5X^{\prime}\Gamma_{n-1}^{-1}X)}.
\lb{lkhd}
\ee
where $\Gamma_{n-1}$ is a function of $(\Phi, \Theta, \Sigma).$
An available likelihood immediately facilitates estimation. Maximum likelihood estimators of $(\Phi, \Theta, \Sigma)$ can be obtained by maximizing the likelihood in terms of the parameters or Bayesian posterior estimates can be obtained based on priors specified on  the range of $(\Phi, \Theta, \Sigma)$

Causality and invertibility are natural conditions that are often imposed on the process while using a VARMA$(p, q)$ to model a multivariate time series. Thus, the likelihood \eqref{lkhd} is often implicitly written with a causal invertible process in mind. In such a situation, it is natural to expect that the estimation procedure would honor the constraints imposed by the model and in particular the estimated VARMA$(p, q)$ process will satisfy the conditions of causality and invertibility.  Such estimators are obtained by maximizing the likelihood \eqref{lkhd} over the parameter space $\mathfrak{P}$  or by specifying a prior distribution over the same parameters space. Thus, in principle likelihood based estimation and inference of causal invertible VARMA processes can be carried out while maintaining the constraints imposed by causality and invertibility.  The reality is however very far from this idealized situation. To our knowledge there exist no current estimation procedure or software package guaranteeing  that the estimated VARMA process is causal and invertible.  

The Schur-Stable space $\mathfrak{S}_k^m$ is essentially described via the roots of the matrix polynomials. The roots are highly non-linear functions of the coefficients $\Phi$ and $\Theta,$ and often are implicitly defined. Thus, maximization of the likelihood \eqref{lkhd} over the parameter space $\mathfrak{P}$ is not a feasible proposition  when optimization is done in terms of the coefficient matrices $\Phi$ and $\Theta$. Overall, direct maximization of \eqref{lkhd} under the constraints on the parameters seems to be a computationally intractable problem. 

 For Bayesian estimation that guarantees causal invertible estimates, one has to specify priors that are fully supported on $\mathfrak{P}.$ Such a prior may not be readily available due to the complexity of the constraint set. The available option seems to be limited.  The obvious, though not entirely satisfactory, choices for prior specification under constraints seem to be the following:\\

\noindent (I) One could specify a flat prior on $\mathfrak{P}$. Even though the roots associated with the polynomials in  $\mathfrak{S}^m_k$ are in bounded sets, it is not immediately clear if such specification provides a proper  posterior for the coefficients. Moreover, as  a Bayesian one would like more flexibility in prior specification. \\

\noindent (II) One could specify a prior over the ambient unrestricted parameter space, draw posterior samples using the unrestricted prior and then discard any sample that do not fall within the constraint set \cite{gelfand_etal}. The advantage of this method is that the induced prior is  fully supported on the constraint space and there are a lot more options for unrestricted priors, e.g. conjugate priors, so that  posterior draws from the unrestricted posterior can be obtained relatively easily. But the major drawback of the procedure is computational efficiency. In terms of the unrestricted prior, if  the measure of the constraint set is small relative to the support of the prior, most of the posterior draws may fall outside the constraint set and will  be discarded, thereby making the posterior sampling extremely inefficient. This is almost certainly the case in the VARMA setting, specially if $m$, $p$ and $q$ are moderately large. \\

\noindent (III) To increase computational efficiency in the truncation method, instead of discarding the samples that fall outside the constrained set, one could project them back to the set, thereby making sure that posterior samples are accumulating at a fast rate; \cite{gunn_dunson,roy_etal2012}.  A difficulty with this procedure is that if the constraint set is not convex the projection may not be uniquely defined.  Regardless, this method accumulates posterior mass at the boundary of the constraint space which is not desirable.  In smaller samples with a relatively small volume of the constraint space, it is possible that all of the posterior mass is  accumulated on the boundary and the posterior estimates maybe inadmissible. There have been illustrations and discussions of scenarios where it is desirable to have priors with no mass on the boundary; \cite{gallindo_garre, chung_etal}. In the VARMA example, having mass on the boundary means that the prior is entertaining non-stationary models even though the true model is assumed to be a stationary causal model.

Thus, it appears that neither direct maximization of the likelihood nor direct prior specification are possible when the parameter space is $\mathfrak{P}.$
We provide a solution to this problem. 
 In particular, we build a   flexible parameterization of the VARMA problem that makes it possible to carry out estimation  under the restriction of causality and invertibility and overcomes most of the difficulties discussed earlier in the section regarding likelihood based inference for causal invertible VARMA processes.

\section{First order vector autoregression}
The vector autoregression  of order one, VAR$(1),$ is arguably the simplest model in terms of parameter estimation and inference in the class of VARMA$(p, q)$ models. To fix ideas and to illustrate the pitfalls of the existing likelihood based methods, we start with the VAR$(1)$ model. Moreover, the tools developed for the VAR$(1)$ case will be directly applicable to the more general VARMA case discussed in the later sections. 
\subsection{Reparameterization of VAR(1)}
Consider an $m$-dimensional  VAR(1) process $\{ X_t\}$: 
\be
X_t = \Phi_1 X_{t-1} + Z_t.
\lb{var1}
\ee
For the purpose of writing a likelihood  we assume the innovations are normally distributed. 
The parameters of interest are the $m\times m$ coefficient matrix $\Phi_1$ and the $m\times m$ covariance matrix $\Sigma$.

The VAR$(1)$ system is causal iff the roots of ${\tilde{\Phi}}(z) = z - \Phi_1$ are all within the unit disk $\mathcal{D}$ or equivalently all eigenvalues, $\lambda_1(\Phi_1), \ldots, \lambda_m(\Phi_1),$  of the matrix $\Phi_1$ are less than one in absolute value. Thus, for the VAR$(1)$, with a slight abuse of notation, we define the parameter space associated with a  causal polynomial to be 
\[ {\mathfrak{S}}^m_1 = \{  \Phi_1 \in \mathbb{R}^{m\times m}:  |\lambda_j(\Phi_1)| < 1, \; j=1,\ldots,m\}.\]
The space ${\mathfrak{S}}^m_1$ forms a submanifold of $\mathbb{R}^{m\times m}$ and is complicated in nature. The space is defined via the eigenvalue restrictions. The eigenvalues are highly nonlinear functions of the elements of $\Phi_1$ and often are not available in explicit form. Thus, the description of the parameter space is extremely complex, involving multiple nonlinear inequality constraints in terms of the $m^2$ elements of $\Phi_1$, where the constraints are implicitly stated; see \cite{mcelroy_findley} for details. This description is generally cumbersome to use in any optimization procedure; nor is it easily generalized to cases with $p>1$.

The key quantity that motivates our reparameterization in the VAR$(1)$  is the observation that $\left( \Gamma(0), \Sigma, \Phi_1\right)$ satisfy the discrete algebraic Riccati  system, 
\be
\Gamma(0) = \Phi_1 \Gamma(0) \Phi_1^{'} + \Sigma.
\lb{riccati}
\ee
It can be shown that any solution of the system for given positive-definite matrices $\Gamma(0) 
\loeeq \Sigma \loe 0$ will be Schur-Stable. 
The equation is satisfied by the Yule-Walker solution in the VAR$(1)$ process, $\Gamma(1)\Gamma(0)^{-1}$, which is known to be Schur-Stable. More generally, the result can be cast in terms of general transformations of symmetric matrices and their relationship to matrix stability. 
For any square matrix $A \in \mathbb{R}^{m\times m}$ and any symmetric matrix $U \in {\mathscr{S}}^m$ define the transformation $S(A, U) : \mathbb{R}^{m\times m} \times {\mathscr{S}}^m \rightarrow {\mathscr{S}}^m,$ by 
\be S(A, U) = U - AUA^{\prime}.
\lb{stein}
\ee
For any fixed $A \in \mathbb{R}^{m\times m},$ the  $A-Section$ of the transformation defined by $S_A(U) = S(A,U)$ is an automorphism on  ${\mathscr{S}}^m $ and is known as the {\it Stein} transformation. 
Thus,  for the VAR$(1)$ with stationary variance $\Gamma(0)$ and innovation variance $\Sigma,$ $S({\Phi_1},\Gamma(0)) = \Sigma.$ The Stein transformation has been extensively studied in the dynamical systems literature in relation to stability of discrete dynamical systems. 
Stein (1952) showed that $S(A,U) \in {\mathscr{S}}^m_{++}$  for some $U \in {\mathscr{S}}^m_{++}$  iff $A \in \mathfrak{S}^m_1.$ Stein's result implies that one could characterize $\mathfrak{S}^m_1$ in terms of matrices in ${\mathscr{S}}^m_{++}$.  For any $M \in {\mathscr{S}}^m_{++}$ , the pre-image $A_M(U) = \{A : S(A, U) = M\}$ is non-empty as long as $U \loeeq M$ but  it need not be a singleton set. In fact, for any $M \in {\mathscr{S}}^m_{++}$ the entire Schur-Stable class can be generated by the pre-images as $U$ varies over the class of positive definite matrices, i.e., $\mathfrak{S}^m_1  =  \bigcup_{U \in {\mathscr{S}}^m_{++}}  A_M(U).$ This is immediate since given $M \in  {\mathscr{S}}^m_{++}$ and $A \in \mathfrak{S}^m_1,$ one can solve for $U$ as 
\[ Vec(U) = (1_{m^2} - A\otimes A)^{-1}Vec(M),\]
where $\otimes$ is the kronecker product. 
 Since the pre-images are non-empty only when $U \loeeq M$, we have 
\be
\mathfrak{S}^m_1 = \bigcup \{A_M(U) : U \in {\mathscr{S}}^m_{++}, \; { U \loeeq M} \}. 
\lb{span_var1}
\ee
For any fixed $M \in {\mathscr{S}}^m_{++}$, the relation \eqref{span_var1} allows us to parameterize the entire Schur-Stable class $\mathfrak{S}^m_1$ in terms of  elements of ${\mathscr{S}}^m_{++}.$  However, since the pre-images $A_M(U)$ are not necessarily singletons, we need to introduce additional parameters that can characterize the pre-images uniquely. 
Before we state our result, we introduce further notation.  Let for $r \leq m,$    $\nu_{r,m}$ denote  the Stiefel manifold of $r\times m$ semi-orthogonal matrices.  In the special case when $r$ equals $m$, the  set is  the orthogonal group ${O}(m)$ of $m\times m$ orthogonal matrices. The special orthogonal group $SO(m)$ is the set of matrices in ${O}(m)$ with determinant equal to one. Then we have the following result that characterizes the set of Schur-Stable matrices.

\begin{proposition}
Let $M  \in {\mathscr{S}}^m_{++}$ be given. Then there exists $A \in \mathfrak{S}^m_1$ and $U \in   {\mathscr{S}}^m_{++}$  such that $S(A,U) = M$  iff $ U \loeeq M$ and $A = (U - M)^{1/2} Q U^{-1/2}$ for some $r\times m$ matrix $Q \in \nu_{r,m}$ where $r = rank(U - M)$, $(U - M)^{1/2}$ is a full column rank square root of $(U - M)$ and $U^{-1/2}$ is a square root of $U^{-1}$.  
\label{prop1}
\end{proposition}

To see how Proposition~\ref{prop1} provides a characterization of $\mathfrak{S}^m_1$, we fix $M \in  {\mathscr{S}}^m_{++}$ and  define $V := V(U)  =  U - M$ for any $U \loe M$ (for the illustration we only consider the full rank case, i.e. $U \loe M$  but parametrization in the case of $U \loeeq M$ will be immediate from the description.) 

Then  from  Proposition \ref{prop1}, we obtain the alternative parametrization of $A \in \mathfrak{S}^m_1$  in terms of  $(V,  Q)$  where $V$ is any positive definite matrix and $Q$ is any orthogonal matrix. Note that the number of free parameters in this parameterization is $\binom{m+1}{2}$ for $V$ and $\binom{m}{2}$ for $Q$. Thus, the total number of free parameters is $\binom{m+1}{2} + \binom{m}{2} = m^2$, the same as that in $A$. More importantly, the transformation $\varphi$ taking $A$ to its pre-parameters $(V(A), Q(A))$ is a bijection between ${\mathscr{S}}^m_{++}\times O(m)$ and $\mathfrak{S}^m_1.$ 

Consider  the map $\varphi$  and its inverse $\vartheta$  (the mappings depend on the choice of $M$, but we will assume that $M$ is fixed throughout and suppress the dependence for notational simplicity) defined as 
\beqn
A &\xrightarrow{\varphi}& (V,Q),\\
 (V,Q) &\xrightarrow{\vartheta}& A.
\lb{maps}
\eeqn
The formulas for $\varphi$ and $\vartheta$ are 
\beqn
  V(A) & = & \sum_{j \geq 1}A^jM A^{\prime j}, \nonumber \\
  Q(A)& = & {\left( \sum_{j \geq 1} A^j M A^{\prime j}
  \right)}^{-1/2}A { \left( \sum_{j \geq 0}A^j M A^{\prime j}
  \right) }^{1/2}, \nonumber \\
A(V, Q) & = & V^{1/2} Q {(V + M)}^{-1/2}.
\lb{maps}
\eeqn

\begin{proposition}
 Let $\varphi$ and $\vartheta$ be as defined in \eqref{maps}. Then $\varphi \circ \vartheta = \mbox{id} = \vartheta
 \circ \varphi$, so that the map is a bijection.  
\lb{prop2}
\end{proposition}

The advantages of having a parameterization of $\mathfrak{S}^m_1$ in terms of  $(V, Q)$ is that we could maximize the likelihood defined over $\mathfrak{S}^m_1$ in terms of $(V, Q),$ or induce prior probability distributions over $\mathfrak{S}^m_1$ through the transformation $\vartheta$ and probability distributions for $V$ and $Q$. We show later in the general VARMA parameterization section how to optimize with respect to $V$ and $Q$ by using further transformations that reparameterize the positive definite matrix  parameter $V$ and the orthogonal matrix parameter $Q$ in terms of unrestricted real numbers.   We end this section by an illustration of the advantage of the parameterization $\varphi$ in terms of prior specification and by contrasting the approach with common prior specification approaches that disregard the constraints of causality. \\

 \subsection{Advantages of constrained prior based on  pre-parameterization}
Consider a two-dimensional Gaussian  VAR(1) process $X_t$ defined by \eqref{var1} with $m=2,$ and parameters 
\be
\Phi = \Phi_0 := \begin{pmatrix} \lambda & 0 \\ 2 &  \lambda \\  \end{pmatrix} \hspace{3mm} \Sigma = \Sigma_0 := \begin{pmatrix} 1 & 0 \\ 0 & 1 \\  \end{pmatrix}.
\lb{maxeigen}
\ee
we will take the roots of the process to be near the causal boundary, $\lambda = 1 - n^{-1}$ to illustrate the effect of unrestricted prior specification. Let $(X_1, \ldots, $  $X_{100})$ be a sample of size $n = 100$ from the process and let \eqref{lkhd}  be the associated likelihood. Let $\pi(\Phi, \Sigma)$ be a prior on the parameters. We will assume apriori $Vec(\Phi) \sim N(\Phi_0, 1_2)$ and independently $\Sigma \sim IW(\Sigma_0, 5+ \kappa)$ where $IW$ denotes the probability density of the inverse-Wishart distribution. The parameterization with $\kappa = 0.5$ produces a reasonably flat prior on $\Sigma$ with finite prior variance and also we center the prior for the coefficient at the true value. The class of prior belongs to the class of the standard normal-inverse-Wishart (NIW) prior speification for Bayesian VAR which includes the popular Minnesota prior \cite{litterman}  as a speicial case. The Bayes estimator of $\Phi$ is then obtained by standard Bayesian computation using the prior and the likelihood \eqref{lkhd}. Let ${\hat{\lambda}}_{1,U}$ denote the largest eigenvalue (in absolute value) of the Bayes estimator of $\Phi$ obtained using the NIW type unrestricted prior.  For the proposed Bayesian procedure with constrained priors, we induce priors on parameters via priors on the pre-parameters $(V,Q, \Sigma).$  The exact formulation is given in the general VARMA section and thus we omit the details here. Let  ${\hat{\lambda}}_{1,C}$  denote the largest eigenvalue (in absolute value) of the Bayes estimator of $\Phi$ obtained using the proposed method. 

Figure~\ref{minnesota}  shows the density histogram of  the maximum eigenvalues, ${\hat{\lambda}}_{1,U}$ and ${\hat{\lambda}}_{1,C}$ based on  400 Monte Carlo replications. The density for the posterior obtained from the unconstrained prior has about 30\% posterior mass outside the causal region while the proposed method is concentrated in the causal region. The left tails of distributions for the two estimators are reasonably close, but the unrestricted estimator has posteior eigenvalues with magnitude bigger than one. 

\begin{figure}[ht]
\centering
\includegraphics[width=120mm]{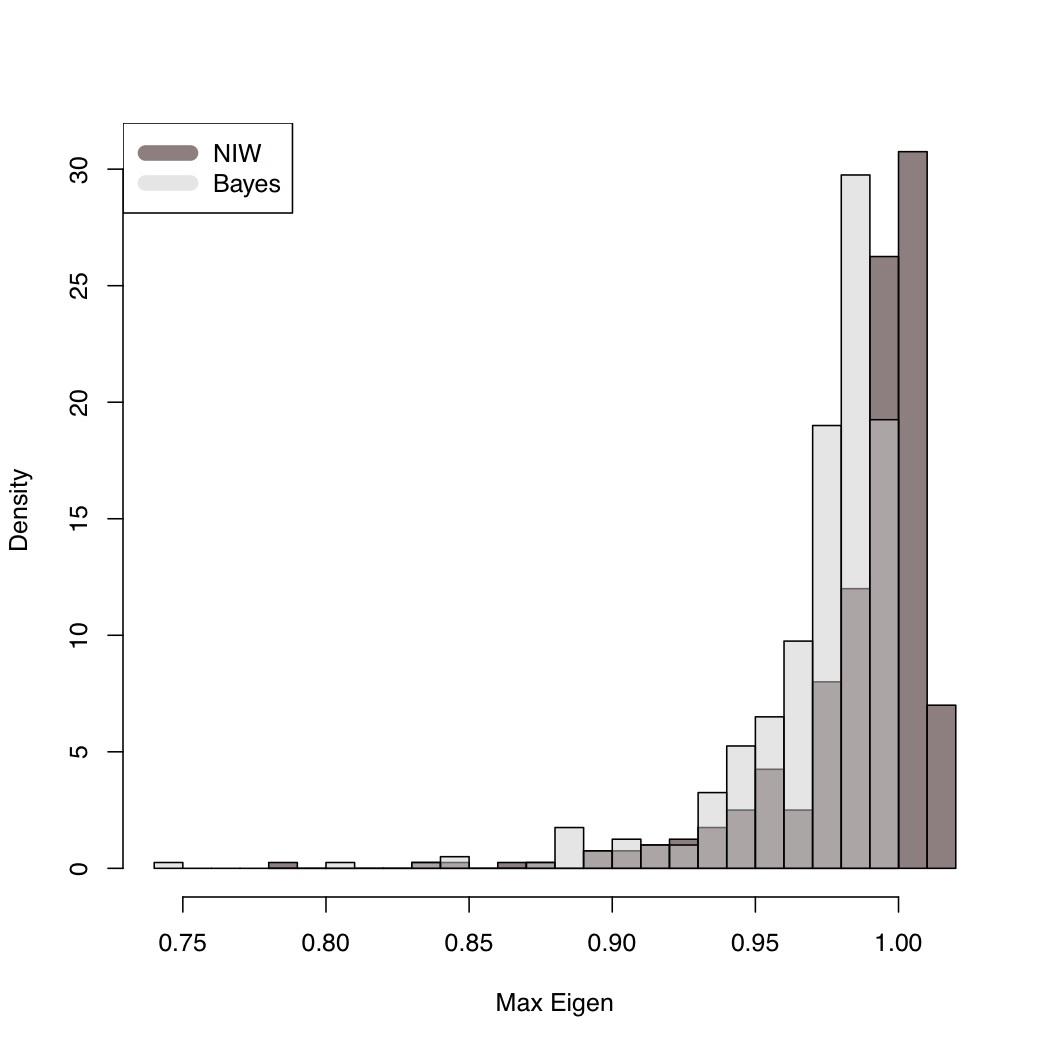}
\caption{Distribution  of the estimated  maximum eigen-value of the coefficient matrix in a VAR(1) model with true values given in \eqref{maxeigen} and sample size equal to 100. The lighter histogram corresponds to the Bayesian estimator based on the proposed parameterization while the darker histogram is for one with standard unrestricted priors on the parameters. The overlapping region between the histograms is shown with an intermediate shade. }
\lb{minnesota}
\end{figure}

\section{Parameterization of Causal Invertible VARMA$(p, q)$}

As discussed in the previous section, causality and invertibility of a VARMA$(p, q)$ process is related to Schur-Stability of the autoregressive and the moving average polynomials. Thus, to parameterize the VARMA$(p, q)$ process in a manner that guarantees causality and invertibility, we first establish characterization of Schur-Stable polynomials in $\mathfrak{S}^m_k$. The characterization is obtained in terms of positive definite  block Toeplitz matrices. the block Toeplitz operators provide a generalization to the Stein transformation used in the VAR$(1)$ example. 

\subsection{Block Toeplitz parameterization}
The following notation will be used hereafter.  For $t \geq 1,$  define ${\underline{U}}_t$ to be a symmetric block Toeplitz matrix defined as 
\be
{\underline{U}}_t = \begin{pmatrix}
U(0) & U(1) & \cdots & U(t) \\
U(1)^{\prime} & U(0) & \cdots & \cdot              \\
\vdots        &  \ddots        & \ddots & \vdots             \\
U(t)^{\prime} & \cdots     & U(1)^{\prime} & U(0)
\end{pmatrix}
\lb{blocktoeplitz}
\ee
where $U(0), U(1), \ldots, U(t)$ are arbitary $m\times m$ matrices and $U(0) \in {\mathscr{S}}^{m} $. Note that ${\underline{U}}_0 = U(0).$  For a block Toeplitz matrix ${\underline{U}}_t$ of the form \eqref{blocktoeplitz}, we define the {\it order} of ${\underline{U}}_t$ to be $t$. 
For $t \geq 1,$ we will use the following nested representations of ${\underline{U}}_t$ in terms of ${\underline{U}}_{t-1}$: The lower representation given by 
\be
{\underline{U}}_t = \begin{pmatrix}
U(0) & \xi_t^{\prime} \\
\xi_t & {\underline{U}}_{t-1}
\end{pmatrix} ,
\lb{lower_ut}
\ee
and the upper representation given by 
\be
{\underline{U}}_t =  \begin{pmatrix}
{\underline{U}}_{t-1} & \kappa_t \\
\kappa_t^{\prime} & U(0)
\end{pmatrix} ,
\lb{upper_ut}
\ee
where $\xi_t^{\prime}  = (U(1), \cdots, U(t))$ and $\kappa_t^{\prime} = (U(t)',\cdots, U(1)').$  We set  $\xi_0$ and $\kappa_0$ equal to zero matrices. The Schur complements of $U(0)$ in ${\underline{U}}_{t}$ in the two representations \eqref{lower_ut} and \eqref{upper_ut} are, respectively,
\be
C_t = U(0) - \xi_t^{\prime}{\underline{U}}_{t-1}^{-1}\xi_t,
\lb{lower_schur}
\ee
and
\be
D_t = U(0) - \kappa_t^{\prime}{\underline{U}}_{t-1}^{-1}\kappa_t.
\lb{upper_schur}
\ee
We  define $C_0 = D_0 = U(0).$ 
Also  let  $\mathfrak{T}^{m,k}$ denote the set of symmetric block Toeplitz matrices with $m$-dimensional blocks and order $k$, i.e.
\[ {\mathfrak{T}}^{m,k} = \{ {\underline{U}}_k \in  {\mathscr{S}}^{mk} : {\underline{U}}_k  \mbox{ is in the form } \eqref{blocktoeplitz} \} \]
Also define $ {\mathfrak{T}}^{m,k}_{++}$  to be the subset of $ {\mathfrak{T}}^{m,k}$ comprising the positive definite block Toeplitz matrices of order $k$ and $m$-dimensional blocks.

 \begin{thm}
  \label{thm:stabCond}
  An $m$-dimensional matrix polynomial $A(z) = z^k - A_1 z^{k-1}$ $  \cdots - A_k$ is Schur-Stable i.e., $A(z) \in \mathfrak{S}^m_k$  iff there exists a block Toeplitz positive definite matrix of the form ${\underline{U}}_{k}$ in \eqref{lower_ut} such that the coefficients $A = [A_1, \ldots, A_k]$ satisfy  $A = \xi_{k}' {\underline{U}}_{k-1}^{-1}.$
 \end{thm}

What Theorem~\ref{thm:stabCond} gives us is a  way of parameterizing Schur-Stable polynomials via positive definite block Toeplitz matrices. However, positive definite block Toeplitz matrices are not necessarily easy to deal with. Representations of block Toeplitz matrices can be obtained in manner similar to those in \cite{constantinescu, delsargte_etal} but these parameterizations are not convenient for optimization or prior specification. Thus, we need further parameterization of $\mathfrak{T}^{m,k}_{++}.$ The following theorem provides a characterization of elements of $\mathfrak{T}^{m,k}_{++}$ that will allow us to parameterize $\mathfrak{T}^{m,k}_{++}$, and hence $\mathfrak{S}^m_k$ in terms of simpler objects that lend themselves to optimization and prior specification more easily.

\begin{thm}
  \label{thm:schur_iter}
A block Toeplitz matrix ${\underline{U}}_k$ of the form \eqref{blocktoeplitz}  is positive definite iff the associated Schur complement sequence defined in \eqref{lower_schur}  satisfies $C_0 \loe  C_1 \loe \cdots \loe C_k \loe 0.$\\
\end{thm}

Let $C_k = M$ be fixed and given. 
 Let $V_t = C_{t-1} - C_{t}$ for $1 \leq t \leq k.$  Following  the proof of Theorem~\ref{thm:schur_iter}, we have
\[
 V_t =  { \left( U (t)' - \kappa_{t-1}^{\prime} {\underline{U}}_{t-2}^{-1} \xi_{t-1} \right) } D_{t-1}^{-1}   { \left( U(t)' - \kappa_{t-1}^{\prime}  {\underline{U}}_{t-2}^{-1} \, \xi_{t-1} \right)^{\prime}  },
\]
 which has the solution
\begin{equation}
 \label{eq:gammat}
U(t)' = \kappa_{t-1}^{\prime} {\underline{U}}_{t-2}^{-1}  \xi_{t-1}  + V_t^{1/2}  Q_t D_{t-1}^{1/2}
\end{equation}
 for some orthogonal matrix $Q_t.$ Here $V_t^{1/2}$ and $D_{t-1}^{1/2}$ are  symmetric square roots of $V_t$ and $D_{t-1},$ respectively. Then \eqref{eq:gammat}  defines the key recursion equation that allows us to solve $U(t),\, t=0,\ldots,k,$ iteratively once the positive definite matrices $M, V_1, \ldots, V_k$ and the orthogonal matrices $Q_1, \ldots, Q_k$ have been specified. At the $t$th stage, all quantities on the right side of   \eqref{eq:gammat}  are known and hence $U(t)$ as well as ${\underline{U}}_t$ can be computed. Subsequently $\xi_t, \kappa_t$ and $D_t$ can be obtained from ${\underline{U}}_t$ through \eqref{lower_schur} and  \eqref{upper_schur}  and then used in the $(t+1)$th iteration. 

\[ C_0 = (C_0 - C_1) + (C_1 - C_2) + \cdots + (C_{k-1} - C_k) + C_k = \sum_{t=1}^{k} V_t + M.\]
Thus,   to initialize the algorithm we can set  $U(0)  = C_0 =  \sum_{t=1}^{k}  V_t + M.$  The base case of (\ref{eq:gammat}) is
\[
U(1) = V_1^{1/2} \, Q_1 D_0^{1/2} =   V_1^{1/2} \, Q_1 U(0)^{1/2}.
\]
The full iteration starts with the $t=2$ case. \\

\noindent {\bf Algorithm for computing $A(z)$ from $V_1,\ldots, V_p, Q_1, \ldots, Q_p$}\\
The following steps describe the algorithm for computing the block Toeplitz matrix ${\underline{U}}_p$ and hence $A,$ the  coefficients of the Schur-Stable polynomial $A(z)$  from the pre-parameters $ V_1, \ldots, V_p, Q_1, \ldots, Q_p:$ 
\begin{enumerate}
\item
Set $ U(0) = C_0 =  M + \sum_{j=1}^p V_j. $
\item
Compute  $U(1)'  = V_1^{1/2} \, Q_1 U(0)^{1/2}$ and obtain ${\underline{U}}_1.$
\item
Compute $\kappa_1, \xi_1, D_1$ based on ${\underline{U}}_1.$ Here $\kappa_1 = \xi_1^{\prime} = U(1)$ and 
$D_1 = U(0) - U(1)'U(0)^{-1}U(1).$
\item
Compute $U(2) = \kappa_1^{\prime}U(0)^{-1}\xi_1 + V_2^{1/2}Q_2 D_1^{1/2}$ and obtain ${\underline{U}}_2$ from $U(0),$ $ U(1), U(2).$
\item 
Obtain $\kappa_2, \xi_2, D_2$ and iterate using \eqref{eq:gammat}.
\item
Once $U(0), U(1), \ldots, U(p)$ and hence ${\underline{U}}_p$  have been obtained compute $A$ using the  Yule-Walker relation $A = \xi_p'{\underline{U}}_{p-1}^{-1}.$
\end{enumerate}
The fact that the parameterization $A(z) \leftrightarrow (V_1,\ldots, V_p, Q_1, \ldots, Q_p)$ is a bijective map from the Schur-Stable space to the space of positive definite matrices and orthogonal matrices can be established in a manner similar to that of Proposition~\ref{prop2}. The algebra is straight-forward but tedious and is omitted. 

\subsection{Generalized Stein transformation}
To draw analogy with the VAR$(1)$ case, we define a generalization of  Stein transformation  on the set of positive definite block Toeplitz matrices. Fix a set of coefficient matrices $A = \left[ A_1,\ldots, A_k\right] \in \mathbb{R}^{m\times mk}$  associated with a polynomial $A(z) = z^k - A_1 z^{k-1} - \cdots - A_k,$ and a symmetric  matrix ${\underline{U}} \in {\mathscr{S}}^{mk}$ with $U_{11}$ as the upper left $m\times m$ block of ${\underline{U}}$.  Let
\be
 \tilde{A} = \begin{pmatrix} A_1 & A_2 & \cdots & A_{k-1}& A_k\\ 1_m & 0 & \cdots & 0 & 0\\
\vdots & \cdots &  \ddots & 0 & \vdots\\
0 & 0 & \cdots & 1_m & 0
\end{pmatrix}.
\lb{atilde}
\ee
 Define the {\it Generalized Stein Transformation} ${\tilde{S}}_k(A, {\underline{U}}) : \mathbb{R}^{m\times mk} \times {\mathscr{S}}^{mk} \rightarrow {\mathscr{S}}^{mk}$ by 
\be
 {\tilde{S}}_k(A,{\underline{U}}) = U - {\tilde{A}} {\underline{U}}{\tilde{ A}}',
\lb{general_stein}
\ee
i.e., ${\tilde{S}}_k(A,{\underline{U}}) = S({\tilde{A}},{\underline{U}}).$ Also define $S_k(A, {\underline{U}}) = U_{11} - A {\underline{U}} A'$ to be the upper left $m\times m$ block of ${\tilde{S}}_k(A, {\underline{U}})$.  
 The Generalized Stein Transformation reduces to the Stein transformation for the case $k  =1.$ 
 Analogous to \cite{stein1952}, one can characterize Schur-Stability of $A(z)$ from properties of the transformation. 
\begin{thm}
\lb{thm:steinStab}
A matrix polynomial $A(z) = z^k - A_1 z^{k-1} - \cdots - A_k$ with coefficients  $A = \left[ A_1, \ldots, A_k\right]$ is Schur-Stable iff there exists a positive definite block Toeplitz matrix ${\underline{U}}_{k-1} \in {\mathfrak{T}}^{m,k-1}_{++},$ such that the Generalized Stein Transform ${\tilde{S}}_k(A, {\underline{U}}_{k-1}) \in {\mathscr{S}}^{mk}_{+}$  with  ${{S}}_k(A, {\underline{U}}_{k-1}) \in {\mathscr{S}}^{m}_{++}. $
\end{thm}
\begin{remark}
In general, for self-dual cones in finite dimensional Hilbert space with a Euclidean Jordan algebra, characterization of Stein-type operators can be done \cite{schneider}. However, due to the special structure of $\tilde{A}$ a more refined result on positivity of the generalized transformation can be obtained.  
\end{remark}

\begin{remark}
The quantity  $S_k(A, {\underline{U}}_{k-1})$  is precisely the Schur-complement $C_k$ under the conditions of  Theorem~\ref{thm:stabCond}. Also, from the proof of Theorem~\ref{thm:steinStab}  it is clear that if ${\underline{U}}_{k-1}$ satisfying the conditions of Theorem~\ref{thm:steinStab} exists, then necessarily ${\tilde{S}}_k(A, {\underline{U}})$ will be of the form
\[ {\tilde{S}}_k(A, {\underline{U}}) = \begin{pmatrix} {S}_k(A, {\underline{U}}) & 0 \\ 0 & 0 \end{pmatrix}.\]
\end{remark}

As in the VAR$(1)$ case, for any fixed $M \in {\mathscr{S}}^{m}_{++},$ consider the pre-image
\[ A_M({\underline{U}}_{k-1}) = \{ A: S_k(A, {\underline{U}}_{k-1}) = M \}. \]
The set is non-empty since for any $M$  by Theorem~\ref{thm:schur_iter} one could construct positive definite block Toeplitz matrix ${\underline{U}}_k$ with $C_k = M$  and  then  by  Theorem~\ref{thm:stabCond} any $A$ of the form $A = \xi_k'{\underline{U}}_{k-1}^{-1}$ will be a member of the pre-image. Also, as before it seems natural to expect that 
\[ \mathfrak{S}^m_k = \bigcup \{ A_M({\underline{U}}_{k-1}): {{\underline{U}}_{k-1} \in \mathfrak{T}^{m,k}_{++}}\}. \]
The result is established by noting that given $A \in \mathfrak{S}^m_k,$ and $M \in {\mathscr{S}}^{m}_{++},$ one can construct a causal VAR$(k)$  model with $A$ as the coefficients and $M$ as the innovation variance and then ${\underline{U}}_{k-1} = \Gamma_{k-1}$ will satisfy the Generalized Stein Transformation. This  reinforces the fact that the class of Schur-Stable polynomials can be parameterized by the class of positive definite block Toeplitz matrices which in turn is done with the help of Theorem~\ref{thm:schur_iter}.

\subsection{Further reparametrization}

From the general algorithm for constructing coefficients of Schur-Stable polynomials in terms of positive definite matrices and orthogonal matrices, we can define a  pre-parameterization of the autoregressive and the moving average polynomials of a causal invertible VARMA$(p, q)$ process as $(V_1^{AR}, Q_1^{AR}, \ldots, V_p^{AR}, Q_p^{AR})$ and $(V_1^{MA}, Q_1^{MA}, \ldots, $ $V_p^{MA}, Q_p^{MA}),$ respectively. This, in addition to the innovation variance $\Sigma$ provides the full pre-parameterization of the  process. To obtain the pre-parameterization one has to fix the $M$ matrix given in the algorithm.  Without loss of generality,the matrix can be chosen to be $M = 1_m$ and from our investigation it is clear that the estimators are not sensitive to the choice of $M$.  We note in passing that our parameterization provides the autocovariance function sequence of the VARMA$(p, q)$ directly, which is useful in applications such as forecasting and prediction.

The main objective of the parameterization is facilitating likelihood based estimation, e.g. maximum likelihood estimation or Bayesian estimation. To that end, the positive definiteness constraints and the orthogonality constraints still pose significant challenges.  Thus, it is convenient to further parametrize the basic pre-parameters $\Sigma$, $V$'s and $Q$'s in terms of essentially unrestricted real numbers.

First we consider the positive definite matrices.  
While the positive definite matrices $\Sigma, V_1, \ldots, V_p$ may be used directly for prior specification, they need not be the most suitable quantities for likelihood optimization. For further simplification one could use the forms described in \cite{lindstrom_bates, leonard_hsu} or  other forms of decomposition given in terms of eigenvalues and eigenvectors can  be pursued.  
The  standard parameterization of positive definite matrices in terms of their Cholesky decomposition given in \cite{lindstrom_bates}  is particularly useful.  For any positive definite matrix the Cholesky form decomposes the matrix in terms of a lower triangular matrix with ones along the diagonal and a diagonal matrix with positive entries along the diagonal. Thus for an $m\times m$ positive definite matrix, the representation provides $\binom{m}{2}$ real numbers as the entries of the lower triangular matrix and $m$ positive real numbers as the diagonal entries, thereby keeping the total number of free parameters equal to that of the positive definite matrix, $\binom{m+1}{2},$ but having the advantage that the parameters of the decomposition are essentially unrestricted. The specific form for an $m\times m$ positive definite matrix $V$ is $V = L D L^{\prime}$ where $L$  and $D$ are given by
\[ L = \begin{pmatrix} 1 & 0  & \cdots & 0\\
l_{2,1} & 1 & \cdots & 0\\
\vdots & \vdots & \ddots & 0 \\
l_{m,1} & l_{m,2} & \cdots &1\end{pmatrix}, \;\;\;\; D =  \begin{pmatrix} e^{d_1} & 0 & \cdots & 0\\
0 & e^{d_2} & \cdots & 0 \\
\vdots & \cdots & \ddots & 0\\
0 & 0 & \cdots & e^{d_m} \end{pmatrix}. \]
This will henceforth be called the LDL decomposition.
The quantities $(l_{2,1}, $ $ \ldots, l_{m,m-1})$ and $( d_1, \ldots, d_m)$ are $\binom{m+1}{2}$ real numbers that parameterize the entries of $V$.

The orthogonal matrices $Q$s are neither suitable for optimization nor for prior specification. Any element in $O(m)$ can be connected to one in $SO(m)$ through a single householder reflection. Let $E_{\delta} = 1_m - 2\delta e_1e_1^{\prime},$  where $\delta \in \{0,1\},$ and $e_1 = (1, 0, \ldots, 0)'.$  The matrices $E_0, E_1$ denote two Householder reflections with $E_0 = 1_m$ as the identity reflection. Then any element $Q \in O(m)$ can be viewed as 
\[  Q = E_{\delta} R, \]
 for some $\delta \in \{0,1\}$ and some $R \in SO(m).$  There are several possible parameterizations  available for $SO(m),$ and along with $\delta$ that provides full parameterization of $O(m).$
To bring in more flexibility in numerical computation involving $Q$, it seems prudent to further decompose $R$ in terms of simpler quantities.

One obvious parameterization of $SO(m)$ is in terms of Given's rotations (angles). However, for higher values of $m$ the decomposition seems cumbersome.We use the Cayley representation of $SO(m)$ in terms of skew-symmetric matrices along with Householder  reflections to generate a parameterization of the entire $O(m)$. The Cayley representation \cite{cayley} says that for any matrix $Q \in SO(m)$ there exists a skew-symmetric matrix $S$ such that 
$Q = (1_m - S)(1_m + S)^{-1}.$

 Then every orthogonal matrix $Q$ in $O(m)$ can be  written as 
\be
 Q(\delta,S) = E_{\delta}(1_m - S)(1_m + S)^{-1},
\lb{cayley}
\ee
 for some skew-symmetric matrix $S$ and some reflection $E_{\delta}.$  The key advantage is that the distinct entries of $S$ are unconstrained, i.e., can be any real numbers; this arrangement does not naturally arise from the Givens formulation. Here we could let the lower triangular elements of $S$ range freely, setting the diagonal to zero and the upper triangular portion equal to the negative transpose of the lower trainagular part. Thus $\binom{m}{2}$ real numbers, as well as a discrete choice of $\delta,$ determines each $Q$.

\section{Numerical computation}
Even without the constraints of causality and invertibility, numerical computation and estimation for autoregressive moving average models with nonzero moving average components have traditionally been quite challenging. This difficulty in computation is due to the non-linearity of the model involving the moving average component. The parameterization described in the previous section provides a feasible way for carrying out likelihood based estimation for parameters of a causal invertible VARMA$(p, q)$ process. Before proceeding we describe our convention for denoting the pre-parameters associated with the autoregressive part and the moving average parts, respectively. 

The  matrix valued pre-parameters associated with the autoregressive part will be denoted as $(V_1^1, \ldots, V_p^1)$ and $(Q_1^1, \ldots, Q_p^1)$ whereas those associated with the moving average part will be denoted as $(V_1^2, \ldots, V_q^2)$ and $(Q_1^2, \ldots, Q_q^2),$ respectively. The innovation variance will be denoted by $\Sigma$. The matrix $M$  needed to start the parameterization of the Schur-stable polynomials will be taken as the identity matrix $1_m$ for both autoregressive and the moving average polynomials. The scalar valued parameters associated with  LDL decomposition of the $V$ matrices for the autoregressive part will be denoted as ${\underline{l}}^{1,j} = (l_{2,1}^{1,j}, \ldots,l_{m,m-1}^{1,j})$ and ${\underline{d}}^{1,j} = (d_1^{1,j}, \ldots, d_m^{1,j})$ for $\;j=1,\ldots,p,$ and as  ${\underline{l}}^{2,j} = (l_{2,1}^{2,j}, \ldots,l_{m,m-1}^{2,j})$ and ${\underline{d}}^{2,j} = (d_1^{2,j}, \ldots, d_m^{2,j}),$ $\;j=1,\ldots,q$ for the moving average part.  Similarly the scalar pre-parameters associated with the orthogonal matrices will be denoted as 
${\underline{s}}^{1,j} = (s_{2,1}^{1,j}, \ldots, $ $s_{m,m-1}^{1,j})$ and ${\underline{\delta}}^{1,j},$  $\;j=1,\ldots,p,$ for the autoregressive part  and as  ${\underline{s}}^{2,j} = (s_{2,1}^{2,j}, $ $\ldots,s_{m,m-1}^{2,j})$ and ${\underline{\delta}}^{2,j}$ for $\;j=1,\ldots,q$ for the moving average part. The LDL parameters associated with the innovation variance are written as ${\underline{l}}^{0} = (l_{2,1}^{0}, \ldots,l_{m,m-1}^{0})$ and ${\underline{d}}^{0} = (d_1^{0}, \ldots, d_m^{0}).$ All together the real pre-parameters are then $(\underline{l}, \underline{d}, \underline{s})$ where $\underline{l} = ({\underline{l}}^{1,1}, \ldots, {\underline{l}}^{1,p},{\underline{l}}^{2,1},\ldots, {\underline{l}}^{2,q}, {\underline{l}}^0)$ and $\underline{d}, \underline{s}$ are similarly defined.

\subsection{Maximum likelihood}
The standard optimization methods are either gradient based stepping algorithms or direct search algorithms. Gradient based methods for the VARMA model may have difficulty due to the  non-linear nature of the likelihood. The direct search methods may fail due to the high number of parameters. As discussed in the introduction, there has been extensive work on VARMA estimation using maximum likelihood. Several procedures have been proposed in the literature  none of which maintain the constraints of causality and invertibility.  For maximum likelihood computation it will be convenient to use the LDL form for the positive definite part of the parameterization. Thus all parameters except the $\delta$'s  e.g. ${\underline{l}}, {\underline{d}}$ and ${\underline{s}},$ are  unrestricted real numbers. For a VARMA$(p, q)$ the $\delta$ parameters ${\underline{\delta}}  = ({\underline{\delta}}^{1,1},  \ldots, {\underline{\delta}}^{1,p},{\underline{\delta}}^{2,1}, \ldots, {\underline{\delta}}^{2,q})$ lie in $\{0, 1\}^{p+q}$ providing $2^{p+q}$ possible values for the parameters. Suppose the VARMA likelihood is written as $\mcL(\Phi, \Theta, \Sigma)$ where it is understood that the parameters are functions of the pre-parameters ${\underline{l}}, {\underline{d}}, {\underline{s}}$ and ${\underline{\delta}}.$ Then the maximum likelihood estimators are defined as 
\[ ({\hat{\Phi}}, {\hat{\Theta}}, {\hat{\Sigma}}) = \underset {\delta \in \{0, 1\}^{p+q}}{\arg \max} \;\underset {{\underline{l}}, {\underline{d}}, {\underline{s}}}{\arg \max}\; \mcL(\Phi, \Theta, \Sigma). \]
Since $2^{p+q}$ will tend to be moderate in practice, a profile approach is recommended, i.e., compute maximizers for each value of $\delta$ and then take the best option.

To initialize the optimization algorithm one could choose arbitrary values of the pre-parameters.  In optimization with a  high number of parameters, the choice of initial values may be important. For a VAR$(p)$ one could initialize at the pre-parameter values associated with the Yule-Walker solution. For a general VARMA model there are no moment based or ad-hoc crude estimators of the parameters that restrict estimation to the causal invertible space. One possibility is to use a truncated Wold representation of the white noise series to write a VAR representation
\[  X_t =  \sum_{j=1}^{r} \Pi_j X_{t-j} + Z_t, \] 
for a moderately large value of $r,$  estimate $\Pi_1, \ldots, \Pi_r, \Sigma$ using a Yule-Walker solution and compute the residuals. Using the residuals from the long VAR representation, one could initialize the VARMA to be able to estimate the parameters using least squares. Such an approach is detailed in (\cite{lutkepohl}, pp-474). Since there is no guarantee that the estimators will be stable, one needs to convert the regression estimator to a stable estimator before it can be used as an initial estimator in the maximum likelihood procedure. For VARMA(1,1),in the case when either the VAR parameters or the VMA parameter has root outside the unit circle,  one could simply shrink the estimator to have maximum eigenvalue less than one, equal to some prespecified value, e.g. 0.999. For a more general VARMA it is not clear how to shrink the eigenvalues to have roots within the unit circle.  

\subsection{Bayesian Prior Specification and Computation}

While the parameterization helps in frequentist estimation of the VARMA model by making the likelihood maximization an essentially unconstrained optimization problem, it also facilitates Bayesian estimation of the models by allowing flexible specification of priors that are fully supported on the causal invertible space. 

For the VAR model, Normal-Inverse Wishart (NIW) distributions are popular choices for  priors on the coefficient matrices and the innovation variance. The Minnesota prior (Litterman 1980) also follows normal  specifications for coefficients of  individual equations in the VAR model.  Prudent choice of the hyperparameters in these specifications can lead to better forecasting properties for the BVAR. However, none of the current prior choices restrict the prior, and thereby the posterior, to the causal invertible space. More specifically, as described in Section~3 , the posterior probability of estimating a model with unit root or roots outside the unit circle remain quite significant when the sample size is small. In the case of a general VARMA model, the problem of prior specification is more challenging.  No obvious choice of prior exists and the  performance of the NIW priors remains unclear.

The parrameterization given in this article  restricts the prior to the causal invertible space for a general VARMA model. For prior specification one could directly use the pre-parameterization based on matrices or use those based on scalar pre-parameters. For the positive definite part of the  matrix valued pre-parameterization there are several options with obvious  prior choices being  Wishart or Inverse-Wishart. For the orthogonal matrices belonging to $SO(m)$ the choices are more limited.   
For prior specification on $SO(m)$ one could specify the uniform prior which will be  proper due to compactness of $SO(m)$.  Other direct prior specification on $SO(m)$ include  the Bingham-von Mises-Fisher (BMF) distribution \cite{hoff} or those involving the Langevin density on $SO(m)$ \cite{chiuso_etal}. Other methods involving representation of $SO(m)$ can also be used \cite{chikuse}. In addition one needs to specify a prior on $\delta$.
By using a reparameterization as $\delta = I(z < 0)$ one could specify a Gaussian prior on the latent quantity $z$. 
For the scalar representation, one could specify independent Gaussian priors on all the pre-parameters.

Bayesian computation under the described prior specification can be carried out easily using a Metropolis-hastings within Gibbs algorithms. There are no obvious simplification that can be incorporated in the standard implementation of a Markov chain Monte Carlo (MCMC)  routine. The initialization of the chains can be done in a manner similar to that for the MLE computation. More description of the implementation along with an R code is available in the supplementary material.

\section{Simulation}

To evaluate the performance of estimators based on the proposed pre-parameterization we conducted a limited simulation. 
The models explored are vector autoregression models. In the VARMA setting there are no available causal invertible estimator that can be compared with the proposed estimator. We compare the performance of the proposed estimator with that of the Yule-Walker estimator in VAR(1)  and VAR(2) in two dimensions, $m=2$ and VAR$(1)$ in three dimensions, $m =3.$  We summarize the performance of the estimators using the Monte Carlo root mean squared error based on $N=500$ Monte Carlo replication of samples of size $n = 100.$

\subsection{Two-dimensional VAR(1)}
First we consider the simplest model in the VARMA$(p, q)$ class, namely a  first order two-dimensional vector autoregression:
\be
X_t = \begin{pmatrix}\Phi_{11} & 0\\  1 & .8\end{pmatrix} X_{t-1} + Z_t,
\lb{sim_model1}
\ee
where $Z_t \stackrel{iid}{\sim} N(0, 1_2).$  The parameter that we vary  is the upper diagonal entry $\Phi_{11}$ which is also one of the eigenvalues of $\Phi_1$ under this parameterization. The values of $\Phi_{11}$ are chosen from the set $ \{-.95, -.95, -.9, -.8, -.6, $  $-.4, -.2, 0, .2, .4, .6, .8, .9,  .95, .99\}.$  
Maximum likelihood estimation is done using the {\it optim} function in R. The initial values of the pre-parameters are chosen to be those associated with the Yule-Walker estimator. We use  box constraints with upper and lower bounds for the real parameters. The bounds are $\pm \mbox{1e}\plus 30$ for the $l$ and $s$ pre-parameters and $\pm \mbox{1e}\plus10$ for the $d$ parameters. We use a tighter bound  for the $d$ parameter to prevent the iterations from generating near singular values of the $V$ matrix. For Bayesian estimation, $N(0, 5)$  prior is specified for all the pre-parameters except $\delta$ which is assigned a $Bernoulli(0.5)$ prior. The Bayesian updates are done with metropolis random walk for the real parameters and  via independent sampling with a jump distribution of $Bernoulli (0.5)$ for $\delta.$  The metropolis chains are of  length 20,000 with a burn-in of 5000.

Tables~\ref{tab1}-\ref{tab3} report the values of $\sqrt{nMSE},$ denoted by  RMSE for the four entries of $\Phi$  where for any generic parameter $\eta$ and its estimator ${\hat{\eta}}$ obtained based on a sample of size $n,$  the $nMSE$ over $N$ Monte Carlo replications is defined as 
\[ nMSE({\hat{\eta}}) = nN^{-1}\sum_{j=1}^N ({\hat{\eta}}_j - \eta)^2. \]
The values reported under the 'Asymptotic' column are the square-root of $n$ times the asymptotic variance of the parameter.
We also looked at the overall RMSE as a function of $\Phi_{11}$ where the overall RMSE is defined as 
$nN^{-1}\sum_{j=1}^N \| {\hat{\Phi}} - \Phi\|^2,$ and $\| \cdot \|$ is the Frobenius norm of a matrix. 
The overall RMSE is plotted as a function of $\Phi_{11}$ in Figure~\ref{fig1}.

From the reported values we see that the likelihood based estimators are performing better than the Yule-Walker estimator, particularly when the largest root is close to unity. All three estimators are causal and have similar bias but the gain in efficiency for the MLE and the Bayes estimator is largely due to reduction in variance. The Yule-Walker estimator of $\Sigma$ is particularly unstable (not reported here) for processes with roots close to the causal boundary. In terms of the overall RMSE, the likelihood based estimators are nearly twice as efficient as the Yule-Walker estimator when the largest root of the VAR coefficient is close to one.

\begin{table*}
\centering
\caption{RMSE of estimators for $\Phi_{11}$ and $\Phi_{21}$ in model \eqref{sim_model1}}
\begin{tabular}{cccccccc}
\hline
&& \multicolumn{3}{c}{$\Phi_{11}$} & \multicolumn{3}{c}{$\Phi_{21}$}\\
\hline
$\Phi_{11}$ & Asymptotic & Yule-Walker   & Bayes & MLE &  Yule-Walker   & Bayes & MLE\\
\hline
  -0.99 &0.321 &0.459 &0.403 &0.361& 0.466 &0.365 &0.341\\
  -0.95 &0.420 &0.503 &0.468 &0.437& 0.540 &0.413 &0.410\\
  -0.90 &0.513 &0.555 &0.555 &0.511& 0.588 &0.553 &0.537\\
  -0.80 &0.649 &0.616 &0.623 &0.603& 0.669 &0.609 &0.613\\
  -0.60 &0.824 &0.807 &0.839 &0.794& 0.885 &0.844 &0.849\\
  -0.40 &0.927 &0.867 &0.894 &0.849& 0.946 &0.923 &0.910\\
  -0.20 &0.982 &0.922 &0.951 &0.926& 1.022 &1.003 &0.996\\
   0.00 &1.000 &0.980 &1.030 &0.966& 0.985 &0.962 &0.954\\
   0.20 &0.984 &1.007 &0.994 &0.951& 1.000 &0.953 &0.952\\
   0.40 &0.937 &0.952 &0.966 &0.924& 0.941 &0.904 &0.896\\
   0.60 &0.861 &0.868 &0.847 &0.815& 0.906 &0.849 &0.852\\
   0.80 &0.750 &0.855 &0.778 &0.758& 0.860 &0.779 &0.773\\
   0.90 &0.678 &0.819 &0.731 &0.684& 0.959 &0.714 &0.714\\
   0.95 &0.636 &0.807 &0.687 &0.619& 1.098 &0.585 &0.586\\
   0.99 &0.598 &0.843 &0.644 &0.649& 1.439 &0.677 &0.681\\
\hline
\end{tabular}
\lb{tab1}
\end{table*}

\begin{table*}
\centering
\caption{RMSE of estimators for $\Phi_{12}$ and $\Phi_{22}$  in model \eqref{sim_model1}}
\begin{tabular}{c c c   c  cccc}
\hline
&& \multicolumn{3}{c}{$\Phi_{12}$} & \multicolumn{3}{c}{$\Phi_{22}$}\\
\hline
$\Phi_{11}$ & Asymptotic & Yule-Walker   & Bayes & MLE  & Yule-Walker   & Bayes & MLE\\
\hline
  -0.99 &0.523 &0.663 &0.562 &0.558 &0.719 &0.573 &0.532\\
  -0.95 &0.521 &0.601 &0.570 &0.555 &0.722 &0.603 &0.567\\
  -0.90 &0.518 &0.569 &0.536 &0.524 &0.682 &0.597 &0.545\\
  -0.80 &0.512 &0.574 &0.539 &0.525 &0.667 &0.610 &0.562\\
  -0.60 &0.497 &0.564 &0.517 &0.516 &0.606 &0.551 &0.504\\
  -0.40 &0.478 &0.533 &0.525 &0.521 &0.603 &0.539 &0.509\\
  -0.20 &0.454 &0.499 &0.466 &0.444 &0.560 &0.491 &0.475\\
   0.00 &0.424 &0.482 &0.477 &0.419 &0.534 &0.466 &0.454\\
   0.20 &0.385 &0.453 &0.435 &0.424 &0.470 &0.400 &0.386\\
   0.40 &0.337 &0.435 &0.414 &0.371 &0.440 &0.388 &0.370\\
   0.60 &0.276 &0.377 &0.317 &0.316 &0.396 &0.301 &0.299\\
   0.80 &0.203 &0.283 &0.257 &0.254 &0.324 &0.223 &0.221\\
   0.90 &0.161 &0.234 &0.202 &0.199 &0.351 &0.199 &0.201\\
   0.95 &0.140 &0.204 &0.174 &0.173 &0.376 &0.159 &0.158\\
   0.99 &0.122 &0.205 &0.170 &0.180 &0.456 &0.159 &0.160\\
\hline
\end{tabular}
\lb{tab3}
\end{table*}

\begin{figure}
\centering
\includegraphics[width=3in]{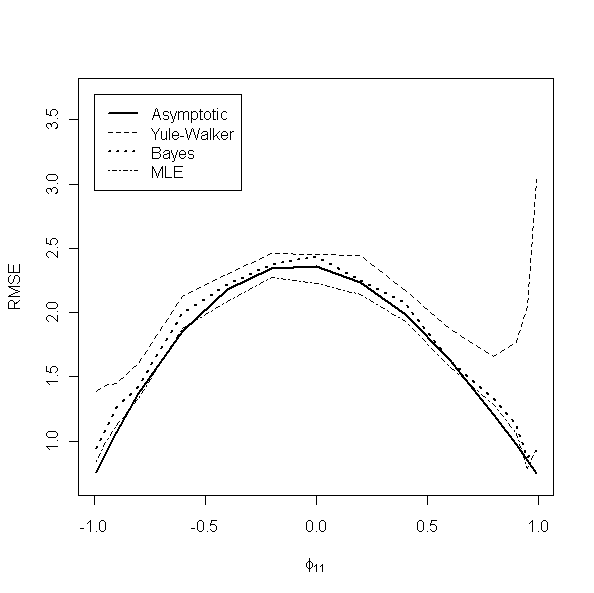}
\caption{Overall RMSE, $nN^{-1}\sum_{j=1}^N \| {\hat{\Phi}} - \Phi\|^2,$  for  different estimators of $\Phi$ compared with the corresponding asymptotic value. The RMSE is plotted as a function of $\Phi_{11}$ in model \eqref{sim_model1}. }
\lb{fig1}
\end{figure}

\subsection{Three-dimensional VAR(1)}
The second setting we consider is a first order three dimensional vector autoregression process. Specifically, we consider the model 
\be
 \begin{pmatrix}X_{t,1}\\X_{t,2}\\X_{t,3}\end{pmatrix} =   
\begin{pmatrix}  
\Phi_{11} & 0    & 0 \\
0.1           & 0.5 & 0  \\
\Phi_{31} & 0.4 & 0.8 \\
\end{pmatrix} \begin{pmatrix}X_{t-1,1}\\X_{t-1,2}\\X_{t-1,3}\end{pmatrix}
+  \begin{pmatrix}Z_{t,1}\\Z_{t,2}\\Z_{t,3}\end{pmatrix}
\lb{sim_model2}
\ee
where the errors $Z_t$ are independent and identically distributed as $N(0, 1_3).$ 
The different scenarios considered are $\Phi_{31} \in \{0.1, 1\}$ and $\Phi_{11} \in \{-.99, -.95, $ $ -0.5, 0.4, 0.9, 0.99\}.$ Maximum likelihood estimation is initialized at the Yule-Walker solution. For the MLE optimization box constraints are used on the pre-parameters and the values of the bounds are same as those in the VAR(1) case. We  use $N(0, 5)$ priors for the real valued pre-parameters and $Bernoulli$ $(0.5)$ prior for the reflection parameter  $\delta$.   Since there are 15 parameters in the model, we only report the overall Monte carlo average of  the estimation error for the coefficient matrix, given by  $N^{-1}\sum_{j=1}^N \|{\widehat{\Phi}}^{(j)}  - \Phi\|, $
where  ${\widehat{\Phi}}^{(j)}$ is the estimator of $\Phi$ based on the $j$th Monte Carlo replication. The overall RMSE is reported in Table~\ref{tab5}. The results  show large efficiency gain for the MLE and the Bayes estimator compared to the Yule-Walker estimator, particularly for processes with roots near the boundary. 

\begin{table*}
\centering
\caption{Overall RMSE, square-root of $nN^{-1}\sum_{j=1}^N \| {\hat{\Phi}} - \Phi\|^2,$  for  different estimators of $\Phi$ for model \eqref{sim_model2}. }
\begin{tabular}{c c   c   c   c  }
\hline
& $\Phi_{11}$ & Yule-Walker  & Bayes & MLE  \\
\hline
\multirow{6}{*}{$\Phi_{31}=.1$} 
& -0.99 & 0.207 & 0.192 & 0.182 \\
& -0.95 & 0.204 & 0.201 & 0.190 \\
& -0.50 & 0.241 & 0.258 & 0.233 \\
&  0.40 & 0.252 & 0.259 & 0.238 \\
&  0.90 & 0.225 & 0.205 & 0.199 \\
&  0.99 & 0.227 & 0.193 & 0.198 \\
\hline
\multirow{6}{*}{$\Phi_{31}=1$} 
& -0.99 &  0.213 & 0.193 & 0.183\\
& -0.95 &  0.207 & 0.200 & 0.190\\
& -0.50 &  0.239 & 0.251 & 0.229\\
&  0.40 &  0.241 & 0.247 & 0.225\\
&  0.90 &  0.235 & 0.197 & 0.189\\
&  0.99 &  0.326 & 0.188 & 0.191\\
\hline
\end{tabular}
\lb{tab5}
\end{table*}

\subsection{Second order VAR}
Next we consider a second order two dimensional vector autoregression process. Specifically, the model is 
\be
\begin{pmatrix}X_{t,1}\\X_{t,2} \end{pmatrix} =
\begin{pmatrix}  \Phi_{1,11} & 0 \\ 1 & 0.4  \\ \end{pmatrix}\begin{pmatrix}X_{t-1,1}\\X_{t-1,2} \end{pmatrix} 
+  \begin{pmatrix} 0 & 0 \\ 0 & .45 \end{pmatrix}\begin{pmatrix}X_{t-2,1}\\X_{t-2,2} \end{pmatrix} 
+ \begin{pmatrix}Z_{t,1}\\Z_{t,2} \end{pmatrix}  
\lb{sim_model3}
\ee
where the errors $Z_t$ are assumed to be i.i.d. $N(0, 1_2). $ This particular parameterization provides a one-dimensional parameterization in terms of  one of the roots ($\Phi_{1,11}$)  of the VAR(2) process and is convenient for illustrating the performance of the estimators as a function of the stability of the process as it changes from very stable to near unit root process.  The scenarios considered are $\Phi_{11} \in \{ .99, -.95, -.9, -.8, -.6, -.4, -.2, 0, .2,  .4, .6, .8, .9, .95, .99 \}.$ 
Maximum likelihood estimation is  initialized at the Yule-Walker solution. The priors are again chosen in a default manner with $N(0, 5)$ priors for the real valued pre-parameters and $Bernoulli(0.5)$ prior for the reflection parameters. Bayesian computation is done using metropolis random walk for the real parameters and independent jumps for the reflection parameters.  Due to large number of parameters  we  only report the Monte Carlo average of  the overall estimation error for the autoregressive coefficient matrices $N^{-1}\sum_{j=1}^N ( \|{\hat{\Phi}}_{1}^{(j)}  - \Phi_1\| + \|{\hat{\Phi}}_{2}^{(j)}  - \Phi_2\|)$
where  ${\widehat{\Phi}}_{1}^{(j)}$ is the estimator of $\Phi_1$ based on the $j$th Monte Carlo replication and ${\widehat{\Phi}}_{2}^{(j)}$ is that for  $\Phi_2.$  The likelihood based estimators continue to enjoy large efficiency gain over the moment-based estimators in the second order process. The advantage of the proposed parameterization is also seen in terms of gain in numerical stability of computation near the causal boundary. 

\begin{table*}
\centering
\caption{Overall RMSE, square-root of $nN^{-1}\sum_{j=1}^N ( \|{\hat{\Phi}}_{1}^{(j)}  - \Phi_1\| + \|{\hat{\Phi}}_{2}^{(j)}  - \Phi_2\|),$ for  different estimators of $\Phi_1$ and $\Phi_2$ for model \eqref{sim_model3}. }
\begin{tabular}{c  c  c  c }
\hline
$\Phi_{11}$ & Yule-Walker & Bayes & MLE\\
\hline
-0.99 & 0.602 & 0.372 & 0.401 \\
-0.95 & 0.421 & 0.368 & 0.394 \\
-0.90 & 0.409 & 0.378 & 0.400 \\
-0.80 & 0.387 & 0.370 & 0.385 \\
-0.60 & 0.389 & 0.366 & 0.383 \\
-0.40 & 0.384 & 0.356 & 0.374 \\
-0.20 & 0.380 & 0.351 & 0.375 \\
 0.00 & 0.395 & 0.367 & 0.391 \\
 0.20 & 0.381 & 0.362 & 0.377 \\
 0.40 & 0.403 & 0.370 & 0.388 \\
 0.60 & 0.385 & 0.349 & 0.363 \\
 0.80 & 0.508 & 0.354 & 0.382 \\
 0.90 & 0.597 & 0.314 & 0.338 \\
 0.95 & 0.747 & 0.348 & 0.355 \\
 0.99 & 0.962 & 0.359 & 0.393\\
\hline
\end{tabular}
\lb{tab6}
\end{table*}


\section{Unemployment rate}
 
For an empirical application we consider time series of unemployment proportions of the population, published by the Bureau of Labor Statistics at \url{http://data.bls.gov/cgi-bin/surveymost?ln}. The series we consider is a multivariate series of Unemployment Rate - White - LNS14000003, Unemployment Rate - Black or African American - LNS14000006, Unemployment Rate - Hispanic or Latino - LNS14000009. Unemployment rate is an important macroeconomic series and has been heavily analyzed. In a well-cited paper, \cite{nelson_plosser} analyzed many macro-economic time series including unemployment rate and since then many analyses and discussions have ensued on stability of the unemployment time series. From an economic point of view the series is generally perceived to be stable, a desirable feature for the health of the economy. There are many subtle features in the unemployment rate series, including variations across different sub-groups such as race and gender or  lead-lag relationship among the subgroups, that may be of interest. While there may be several exogenous variables that affect and could help in modeling the time series behavior of the series, here we concentrate on VAR models for the three dimensional unemployment rate series. 

The series are all monthly series starting at January of 1974 and ending at  December of 2013. Thus, the series length is  $n=480$ for each of the three series.  While the mean level of the series may be of interest, here we concentrate on the deviations from the mean. The time series plot of the three series (with individual series mean subtracted from each series) is shown in Figure~\ref{unemployment_plot}. It is clear from the plot that the three series are correlated.  The ACF and the PACF plots of the three series are shown in Figure~\ref{acf_pacf_plot}.  The PACF of the demeaned series have significant correlations at lag 1 and 2, demonstrating that a VAR(2) model may be a reasonable fit.

\begin{figure}[ht]
 \centering
 \includegraphics[width=120mm]{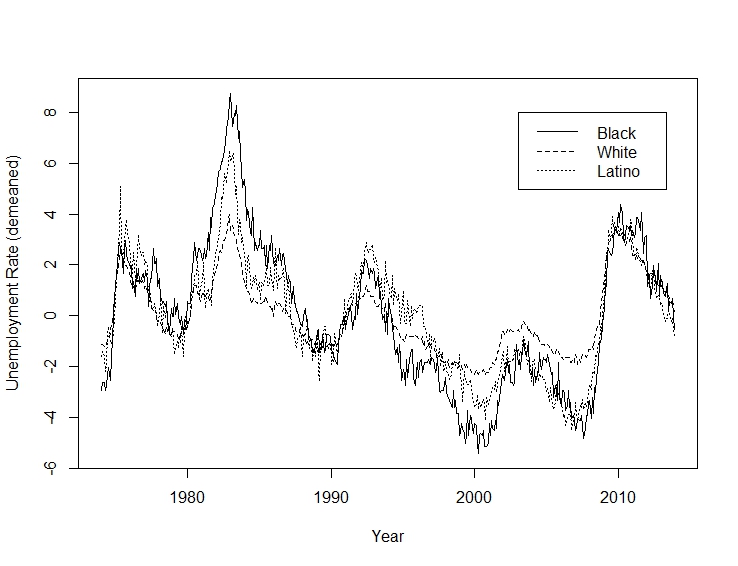}
 \caption{Time series plot of the three unemployment rate series. Each series is demeaned with respect to their individual series mean. }
\lb{unemployment_plot}
 \end{figure}
 
 
\begin{figure}[ht]
 \centering
 \includegraphics[width=120mm]{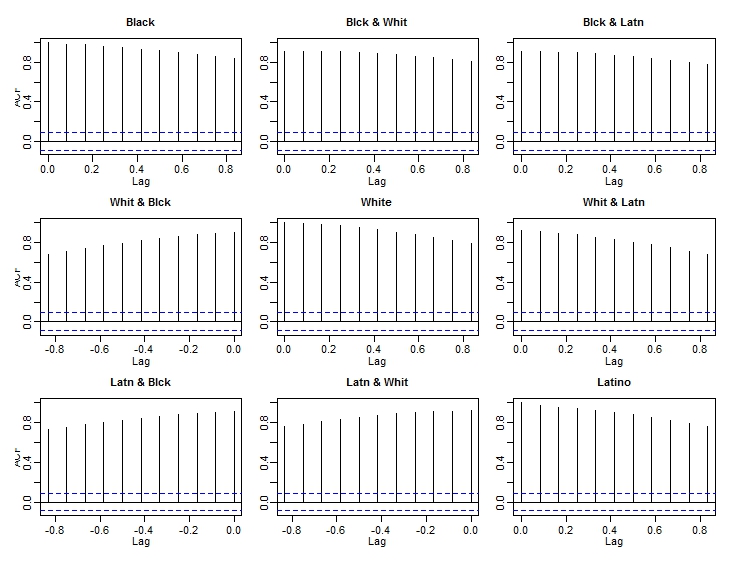}\\
 \includegraphics[width=120mm]{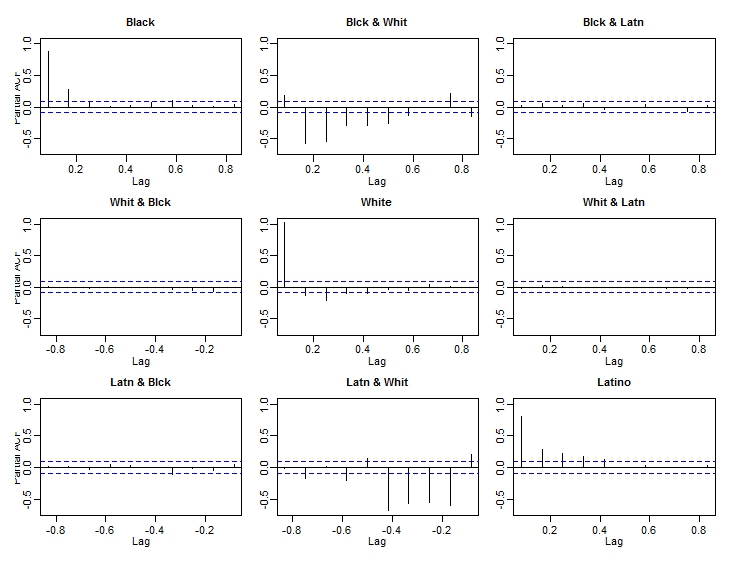}
 \caption{ACF (top 3 rows) and PACF (bottom 3 rows) plots of the demeaned unemployment rate series. The lags are shown in units of the  maximum lag  length  which is equal to 10. }
\lb{acf_pacf_plot}
\end{figure}
 
We fit  a VAR(2)  model using Yule-Walker, the proposed Bayesian and  maximum likelihood estimators based on pre-parameters and also a standard BVAR based on NIW prior on the coefficients and the innovation variance.
 The MLE methodology is initialized at the Yule-Walker estimates, transformed in terms of the  pre-parameters. The Bayesian estimation is done with priors $N(0,5)$ for the real pre-parameters and $Bernoulli(0.5)$ on the reflection parameters.  The NIW prior on the coefficients and the innovation variance is specified as independent $N(0, 0.5)$ for the entries of the coefficient matrices and $IW(1_3, 5)$ for the innovation variance. The point estimates we obtain through the four fitting methods are:
 \[ \widehat{\Phi}_1^{YW} = \begin{pmatrix} 0.632 & 0.727 & -0.034 \\
 0.062 & 1.146& -0.036 \\ 0.0854& 0.756 & 0.567\\ \end{pmatrix} \;\;\; \widehat{\Phi}_2^{YW} = \begin{pmatrix} 0.279 & -0.582 & 0.051 \\ -0.064 & -0.131 & 0.021 \\ -0.054 & -0.609 & 0.286 \\ \end{pmatrix}, \]
\[ \widehat{\Phi}_1^{MLE} = \begin{pmatrix} 0.637 & 0.732 & -0.036\\ 0.066& 1.144 &-0.035 \\ 0.096 & 0.768 & 0.570 \\ \end{pmatrix} \;\;\; \widehat{\Phi}_2^{MLE} = \begin{pmatrix} 0.274 & -0.586 & 0.052 \\ -0.067 & -0.129 & 0.019 \\ -0.064 & -0.623 & 0.284\\ \end{pmatrix}, \]
\[ \widehat{\Phi}_1^{Bayes} = \begin{pmatrix} 0.646 &0.629 &-0.030 \\ 
0.066 &1.149 &-0.034 \\ 0.098 &0.795 & 0.575\\ \end{pmatrix} \;\;\; \widehat{\Phi}_2^{Bayes} = \begin{pmatrix} 0.270 & -0.501 & 0.048 \\ -0.067 & -0.140 & 0.021 \\ -0.068 &-0.676& 0.293 \\ \end{pmatrix}, \]
\[ \widehat{\Phi}_1^{NIW} = \begin{pmatrix} 0.625 & 0.634 & -0.018 \\ 
0.051 &1.154 &-0.047 \\ 0.097 &0.709 & 0.581 \\ \end{pmatrix} \;\;\; \widehat{\Phi}_2^{NIW} = \begin{pmatrix} 0.286 &-0.494& 0.038 \\ -0.051& -0.143& 0.033 \\ -0.064 &-0.571 &0.277 \\ \end{pmatrix}. \]

\begin{table*}
\centering
 \caption{Absolute value of the estimated roots (in decreasing order) of the determinantal equation for the different estimation methods}
 \begin{tabular}{ l  c  c  c  c  c c }
 \hline
 Yule Walker & 0.975 &   0.975 & 0.885 &  0.383 &  0.276 & 0.172 \\
 MLE             & 0.975 &   0.975 & 0.884 &  0.382 &  0.274 & 0.176 \\
 Bayes         & 0.975 &   0.975 & 0.886 &  0.367 &  0.277 & 0.171 \\
 NIW            & 0.990 &   0.961 & 0.894 &  0.363 &  0.252 & 0.144 \\
 \hline
 \end{tabular}
\centering
 \lb{tab7}
 \end{table*}

The roots of the determinantal equation of  the associated VAR(2) polynomial, ${\tilde{\Phi}}(z) = z^2 - \Phi_1 z - \Phi_2, $ allow us to check the stability of the estimates and we do find that all four methods estimate a model in the causal parameter space. Table~\ref{tab7} shows the magnitude of the roots (in decreasing order) of the determinantal equation for the four different fitting methods. As evident from Table~\ref{tab7}, there are  several roots close to the boundary of the stationary parameter space. The unrestricted NIW prior yields a root that is very close to unity. This could be especially problematic for long-term forecasts. There are  interesting economic implications of these fits since there seems to be a moderately strong lag-lead relationship between the White and the other two series. Such relationships could be helpful for prediction purposes. \\

{\it Analysis based on a more recent subset of the data:}
To demonstrate the danger of applying methodologies that do  not  guarantee estimates to be causal,  we examine the  unemployment data from the subset January 2006 to January 2010. This series is considerably shorter and has considerably more unstable features due to the recession and the post-recession recovery. We compare the Bayesian estimates based on the unrestricted NIW prior and the constrained prior based on the proposed pre-parameterization. To allow for possible stable estimates based on the NIW prior, the  prior  means for all $\Phi$ estimates are centered at 0.  Each of the 18 different entries of $\Phi_1$ and $\Phi_2$ are assigned an independent $N(0, 0.5)$ prior. The pre-parameter based estimation is obtained following the same setting as in the simulation set up. The estimated coefficients for the VAR(2) models computed based on the unrestricted and the restricted Bayes methods are: 
\begin{eqnarray}
\;\;\;\;\widehat{\Phi}_1^{NIW} &=& \tiny{ \begin{pmatrix}
0.576 & 0.128 & 0.182 \\
0.036 & 1.135 & 0.036 \\
0.261 & 0.708 & 0.559
\end{pmatrix}, \;\;\;
\widehat{\Phi}_2^{NIW}  = \begin{pmatrix} 
0.060 & -0.036 & 0.153 \\
-0.133 & -0.112& 0.046 \\
-0.395 &-0.683 & 0.556\end{pmatrix}, }
\lb{noncausal_model}
\\
\;\;\;\;\widehat{\Phi}_1^{Bayes} &=& \tiny{ \begin{pmatrix}
0.676 & 0.070 & 0.166 \\
0.053 & 1.353 & -0.082 \\
0.281 & 0.596 & 0.464 \\
\end{pmatrix},  
\widehat{\Phi}_2^{Bayes} =  \begin{pmatrix} 
-0.012 & 0.004 & 0.139 \\
-0.115 & -0.359 & 0.137 \\
-0.373 & -0.571 & 0.592 \\ \end{pmatrix}.}
\lb{causal_model}
\end{eqnarray}

\begin{table*}
\centering
 \caption{Absolute value of the estimated roots (in decreasing order) of the determinantal equation for the Bayesian methods based on the constrained prior specified in terms of the pre-parameters and the unconstrained NIW prior specified directly on the original parameters. One of the estimated roots for the NIW prior is outside the causal region (in boldface). }
\begin{tabular}{ l  c  c  c  c  c  c }
\hline
Bayes & 0.978        & 0.932 & 0.932 &  0.458 &  0.208 &  0.096 \\
NIW   & {\bf1.003} & 0.966 & 0.966 &  0.457 &  0.155 &  0.050 \\
\hline
\end{tabular}
\lb{tab8}
\end{table*}
 
Table~\ref{tab8} shows the magnitude of the roots (in decreasing order) of the determinantal equation associated with the two Bayesian estimates.  The unrestricted NIW  estimates a non-causal model, even though most of  the prior mass is within the causal region. The largest root associated with the NIW prior is 1.003. Shorter-term forecasts, e.g. one-step ahead forecast, will tend to be close for the two estimated model because the models are relatively similar (in absolute term) with respect to the magnitude of the estimated parameters.   However, the models are significantly different in terms of longer-term behavior. Thus, from a long-term policy making perspective of the two models, one may be preferred over the other. The non-causal model  will allow the possible values of unemployment rate to wander substantially away from the mean level, as we demonstrate next.

The risk of maintaining the non-causal model as a long-term dynamic model for the unemployment series becomes apparent if we  investigate the nature of the process implied by the two models. To do so, we generate pseudo data  from the estimated models for the period January 2014 to December, 2023 (10 years of data, 120 observations). For data generation we use November and December, 2013 values as the initial values. The innovation variance estimates from the two models are very similar and we use their mean value 
\[ {\hat{\Sigma}} = \begin{pmatrix} 
0.229 & 0.023 &  0.055\\
0.023 & 0.029 &  0.041\\
0.055 & 0.041 &  0.242
\end{pmatrix}
\]
for data generation. We generate 10,000 replicates and look at the extremes(minimum and maximum) of the 120 observations for each of the replicate.  Figure~\ref{stability_figure} shows the distribution of the extremes for the two models. As evident from the figure, the non-causal model is more unstable. The non-causal model allows more unrealistic values, mainly because the sample paths tend to wander much further away from the mean level of the previous period. 

\begin{figure}[ht!]
\centering
\includegraphics[width= 60mm, height = 55mm]{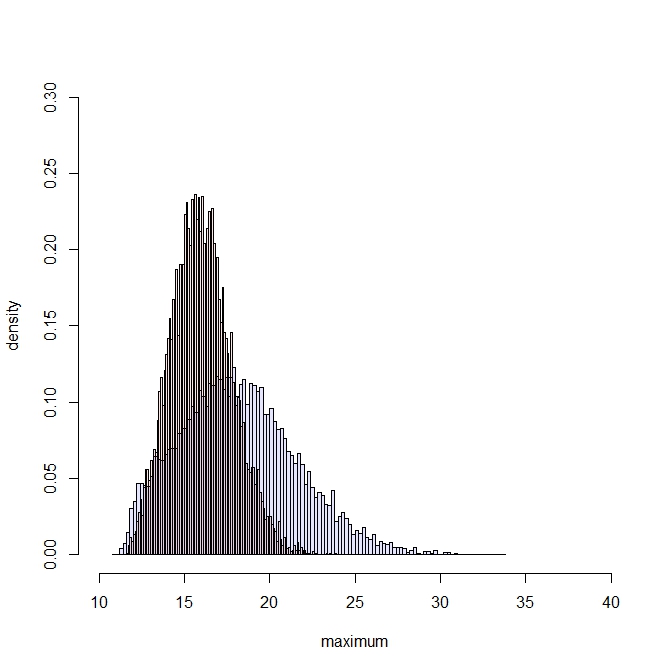}
\includegraphics[width= 60mm, height = 55mm]{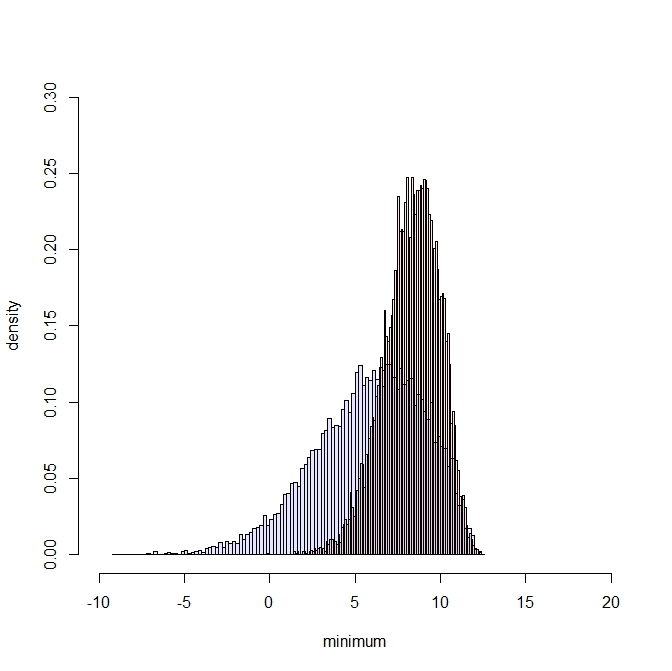}\\
\includegraphics[width= 60mm, height = 55mm]{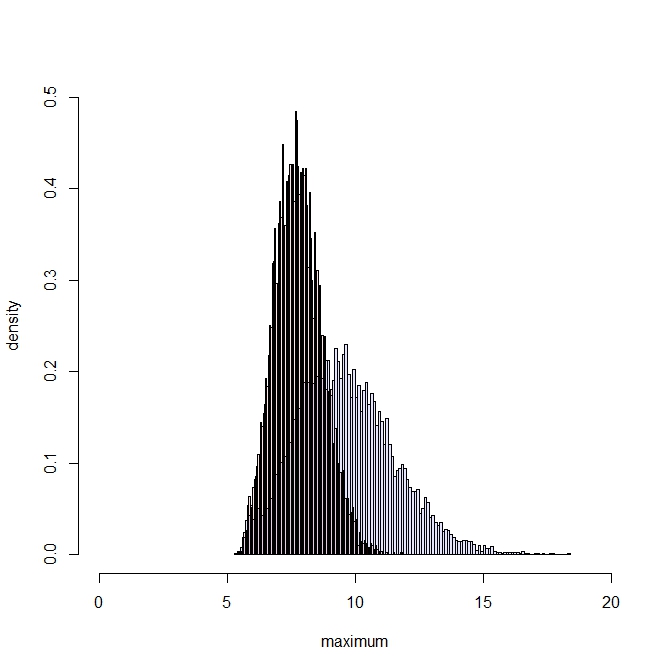}
\includegraphics[width= 60mm, height = 55mm]{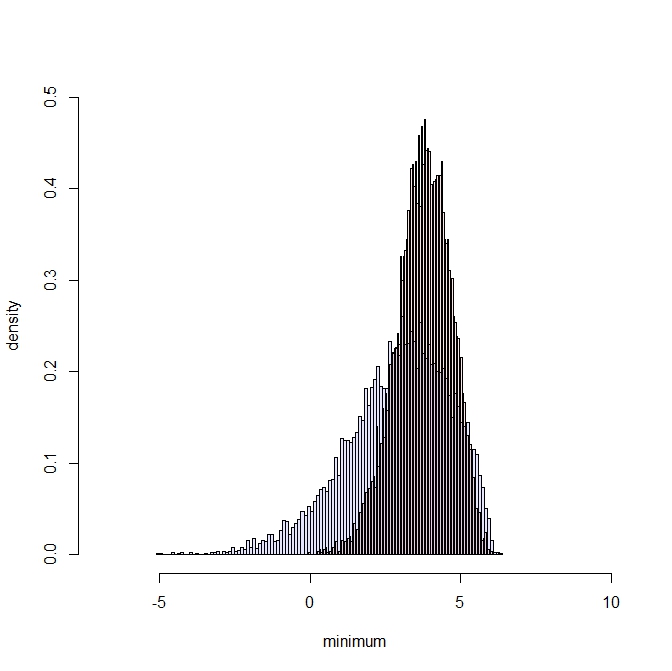}\\
\includegraphics[width= 60mm, height = 55mm]{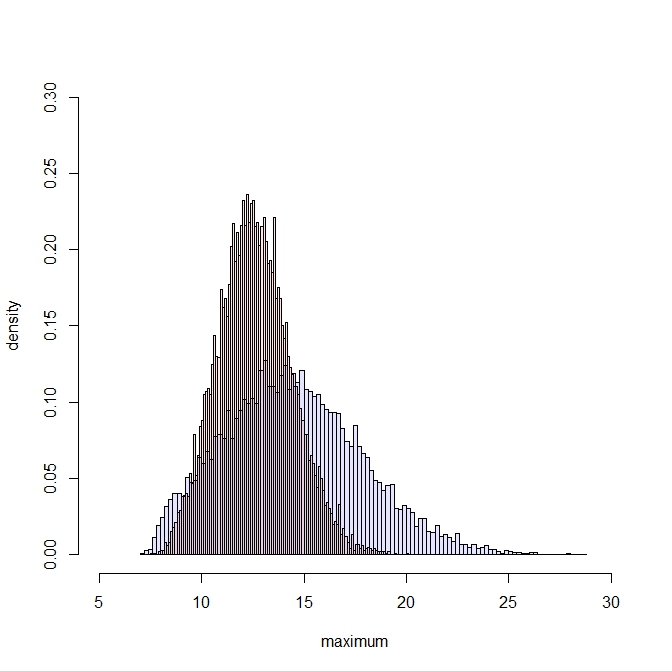}
\includegraphics[width= 60mm, height = 55mm]{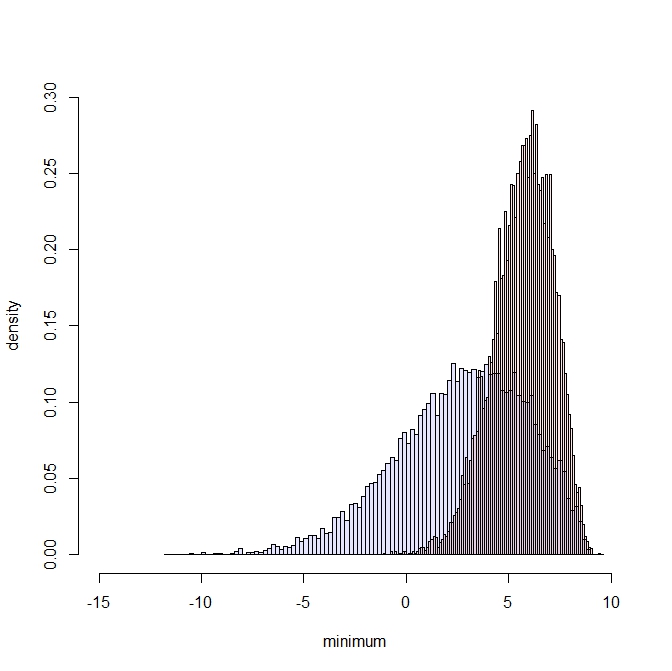}\\
\caption{Each row shows the distribution of the maximum and the minimum of the  pseudo unemployment rate values  generated based on models  \eqref{noncausal_model} and  \eqref{causal_model}. The rows correspond to  blacks, whites and latinos, respectively.In each figure the darker histogram pertains to the causal model \eqref{causal_model} while the lighter corresponds to the noncausal model \eqref{noncausal_model}. The overlapping area  is also visible in the histograms.  }
\lb{stability_figure}
\end{figure}

\section{Discussion}

In this article we introduced a new parameterization of the causal invertible VARMA process that describes the entire class of causal invertible VARMA in terms of unrestricted real-valued parameters. It is apparent from the theoretical and empirical results that the proposed parameterization holds distinct advantages over anything currently available for estimating causal invertible VARMA processes.  The parameterization can be applied to moment-based estimation as well. One could consider an objective function based on the closeness of the sample autocovariances and the theoretical autocovariances written in terms of the pre-parameters. Moreover, one may also consider objective functions in the spectral domain, such as the integrated Frobenius norm of the difference between the observed periodograms and the spectral matrix (written in terms of the pre-parameters). The advantage of the proposed parameterization is that any moments based estimators, obtained as minimizers of objective functions over the pre-parameter space, is guaranteed to be causal and invertible.  

Beyond the obvious advantage of being able to estimate a causal invertible process, the proposed parameterization can potentially facilitate other aspects of VARMA modeling, such as reduced rank formulation \cite{velu_etal1986, ahn_reinsel}. The reduced rank version of the proposed parameterization can be obtained by for the first order polynomial in a straight-forward manner. For higher order polynomials, the exact formulation needs to be investigated.

The proposed parameterization is as dense as the original VARMA parameterization. The total number of parameters is the same as that of the original $m$-dimensional VARMA($p, q$) process, i.e. equal to $(p+q)m^2 + \binom{m+1}{2}.$ For moderate $p,q,$ and $m$ there could be a large number of parameters and as with most multi- dimensional problems with dense parameterization, care needs to be exercised in computation. Sparse reparameterization  that maintain causality and invertibility of higher dimensional VARMA processes is a topic of future research.

\section*{Disclaimer:}
This report is released to inform interested parties of research and to encourage discussion.
The views expressed are those of the authors and not necessarily those of the U.S. Census Bureau.

\appendix
\section*{Appendix: Proofs }\label{app}

\begin{proof}[Proof of Proposition~\ref{prop1}]
The ' if' part follows immediately from substitution. For the converse note that 
\[ (U -M)^{1/2}((U -M)^{1/2})' = (AU^{1/2})(AU^{1/2})'.\]
Therefore we can find  $Q_1 \in \nu_{r,m}$ such that $(U -M)^{1/2} = A U^{1/2} Q_1^{\prime}.$ Hence the result.

\end{proof}

To prove Proposition~\ref{prop2}, we first prove a lemma.

\begin{lemma}
\label{lem:riccatiSoln}
 If $ M \in \mathscr{S}_{+}$, then
 $V =  \sum_{j \geq 1}A^j M A^{\prime j}$ satisfies
 $V = A (V +M) A^{\prime}$.  Moreover, if $A \in
\mathfrak{S}^m_1$, this $V$ is the unique solution to the Riccati
 equation.
\end{lemma}

\begin{proof} The first assertion is trivial algebra.  Now suppose there are two
 solutions, $V$ and $\widetilde{V}$.  Then
\[
A (V - \widetilde{V}) A^{\prime}
 =A ( [V +M] - [\widetilde{V} +M])A^{\prime}
 =A  [V +M]A^{\prime} - A [\widetilde{V} +M]A^{\prime}
 =V - \widetilde{V},
\]
 and by taking $\mbox{vec}$ we obtain
\[
  \left( 1_{m^2} -A \otimes A \right) \, \mbox{vec} (V  -
  \widetilde{V} ) = 0.
\]
 Because $A \in \mathfrak{S}^m_1$, the matrix $1_{m^2} - A\otimes A$
 is invertible, implying that $ \mbox{vec} (V -
  \widetilde{V} ) = 0$, i.e., $V = \widetilde{V}$.
\end{proof}

\begin{proof}[Proof of Proposition~\ref{prop2}]
First consider any $A$ in the domain of $\varphi$, which maps to
\[
 \left( \sum_{j \geq 1} A^j M A^{\prime j},
 {\left( \sum_{j \geq 1} A^jM A^{\prime j}  \right)}^{-1/2} A { \left( \sum_{j \geq 0} A^j M A^{\prime j}
  \right) }^{1/2} \right).
\]
  Applying $\vartheta$ to this yields
\[
  { \left( \sum_{j \geq 1} A^j M A^{\prime j}
  \right) }^{1/2} { \left( \sum_{j \geq 1} A^j M A^{\prime j}
  \right) }^{-1/2} \Phi_1 { \left(  \sum_{j \geq 0} A^j M A^{\prime j}
  \right) }^{1/2} { \left(  \sum_{j \geq 0} A^j M A^{\prime j}
  \right) }^{-1/2},
\]
 which equals $A$; therefore $\vartheta \circ \varphi =
 \mbox{id}$.  The converse follows from Lemma~\ref{lem:riccatiSoln}.

Consider any $(V, Q)$ in the domain of $\vartheta$, which
is mapped to $ V^{1/2} Q $ ${(V + M)}^{-1/2} $
 and we use the abbreviation ${\widetilde{A}}_1 = V^{1/2} Q {(V +
M)}^{-1/2}$ for convenience.  Note that ${\widetilde{A}} (V +
M) {\widetilde{A}}^{\prime} = V$ by algebra.  Letting $U = V +
M$, it then follows that $U ={\widetilde{A}}U {\widetilde{A}}_1^{\prime} +
M.$ Thus by Stein's result we must have ${\widetilde{A}}_1 \in
 \mathfrak{S}^m_1$.  Then by Lemma \ref{lem:riccatiSoln} we know that
 $V$ must be a unique solution to ${\widetilde{A}} (V +
M) {\widetilde{A}}^{\prime} = V$.  Applying $\varphi$ we
 obtain $(\widetilde{V}, \widetilde{Q})$, where
\begin{align*}
 \widetilde{V} & = \sum_{j \geq 1} {\widetilde{A}}^j M
{\widetilde{A}}^{j \prime} \\
 \widetilde{Q} & = \widetilde{V}^{-1/2} {\widetilde{A}} {(\widetilde{V} +
M)}^{1/2}.
\end{align*}
 Noting that $\widetilde{V}$ also solves ${\widetilde{A}} (\widetilde{V} +
M){\widetilde{A}}^{\prime} = \widetilde{V}$ by algebraic
 verification, Lemma \ref{lem:riccatiSoln} tells us that
 $\widetilde{V} = V$.  Plugging this result back into the formula
 for $\widetilde{Q}$ yields
\[
 \widetilde{Q} = V^{-1/2} V^{1/2} Q {(V +M)}^{-1/2} {(V +
M)}^{1/2} = Q.
\]
 Therefore $\varphi \circ \vartheta = \mbox{id}$ as well.
\end{proof}

\begin{proof} [Proof of Theorem~\ref{thm:stabCond}]
Suppose $A(z)$ is Schur-Stable. Then $$A(B^{-1})B^k X_t = Z_t$$ defines a Causal $VAR(k)$ process [Brockwell and Davis, 1996]. Then if we let ${\underline{U}}_{k}  = \Gamma_{k}$ we have the assertion. \\
Conversely, suppose  there is a positive definite block Toeplitz matrix $\underline{U}_k$ of the form \eqref{lower_ut} such that the coefficients of the polynomial $A =  \left[ A_1, \cdots, A_k \right] $  satisfy $A =  \xi_{k}' {\underline{U}}_{k-1}^{-1}.$
The determinantal equation $det(A(z) = 0$ has the same roots as that of $det(\tilde{A}) = 0$ where 
${\tilde{A}}$ is defined in \eqref{atilde}.
We can see that ${\tilde{A}}$ satisfies the system of equations
\be
 {\underline{U}}_{k-1} = \tilde{A} {\underline{U}}_{k-1} {\tilde{A}}' + {\tilde{\Sigma}} 
\lb{var1_form}
\ee
where ${\tilde{\Sigma}} = \begin{pmatrix}C_k & 0 \\ 0 & 0 \end{pmatrix}.$
Because ${\underline{U}}_k$ is positive definite, so is $C_k$. 
 We then obtain stability of ${\tilde{A}}$ modifying the argument of Stein (1952) slightly to show that as long as $C_k$ is positive definite, the eigenvalues of ${\tilde{A}}$ are strictly smaller than one in absolute value.  Let $w^* = (w_1^{*},\cdots, w_p^{*})^{\prime}$ be any left  eigenvector of $\tilde{A}$ corresponding to an eigenvalue $\lambda$. First we show that $w_1 \ne 0.$  From the structure of $\tilde{A}$ we have $w_1^*A_{i-1} +  w_i^* = \lambda w_{i-1}^*$ for $1 < i < k$ and $w_1^* A_k = \lambda w_k^*$. Thus, if $w_1$ is zero then the entire eigenvector is zero leading to a contradiction. Pre and post multiplying equation \eqref{var1_form} by $w^*$ and $w$, we have 
\[ (1 - |\lambda|^2)w^*{\underline{U}}_{k-1} w = w^*\tilde{\Sigma} w = w_1^*C_k w_1 > 0.\]
From the positive definiteness of ${\underline{U}}_{k-1}$ it follows that $|\lambda| < 1$.

\end{proof}

\begin{proof} [Proof of Theorem~\ref{thm:schur_iter}] Let  ${\underline{U}}_k$ be a positive definite block Toeplitz matrix. 
For $2 \leq t \leq k,$
\begin{align*}
 C_{t} & = U(0) - \xi_t^{\prime} \, {\underline{U}}_{t-1}^{-1} \, \xi_t \\
  & =  U(0) - [ \xi_{t-1}^{\prime} , U(t)' ]  \begin{pmatrix}
{\underline{U}}_{t-2} & \kappa_{t-1} \\
\kappa_{t-1}^{\prime} & U(0)
\end{pmatrix}^{-1} \left[ \begin{array}{l} \xi_{t-1} \\ U(t) \end{array} \right] \\
  & = U(0) - \xi_{t-1}^{\prime}{\underline{U}}_{t-2}^{-1} \, \xi_{t-1}   \\
& \qquad \qquad  -  { \left( U(t)' -  \kappa_{t-1}^{\prime}  {\underline{U}}_{t-2}^{-1} \xi_{t-1} \right) }  D_{t-1}^{-1}{ \left( U(t)' - \kappa_{t-1}^{\prime} {\underline{U}}_{t-2}^{-1}  \xi_{t-1} \right)^{\prime} } \\
  & = C_{t-1} - { \left( U(t)' - \kappa_{t-1}^{\prime} {\underline{U}}_{t-2}^{-1} \, \xi_{t-1} \right) } \, D_{t-1}^{-1} \,  { \left( U(t)' - \kappa_{t-1}^{\prime} {\underline{U}}_{t-2}^{-1} \, \xi_{t-1} \right)^{\prime} }.
\end{align*}

Here terms involving $\xi_0, \kappa_0$  for the $t=1$ case are assumed to be zero. 
The second term on the right side is a positive definite term because ${\underline{U}}_k > 0$, which implies that any principal diagonal block is positive definite and hence $D_t >_{L} 0$ for $t \leq k.$ Thus, 
\be 
C_{t-1} - C_t = { \left( U(t)' - \kappa_{t-1}^{\prime} {\underline{U}}_{t-2}^{-1} \, \xi_{t-1} \right) } \, D_{t-1}^{-1} \,  { \left( U(t)' - \kappa_{t-1}^{\prime} {\underline{U}}_{t-2}^{-1} \, \xi_{t-1} \right) }^{\prime} >_{L} 0.
\lb{diffschur}
\ee
Based on the assumption that ${\underline{U}}_k > 0$ we have  $C_k > 0.$  Hence $C_0 > C_1 > \cdots > C_k > 0.$

For the converse, let $C_0 > C_1 > \cdots > C_k > 0$ be defined as in (15). Then $C_0 = U(0) > 0$ and $ C_k = U(0) - \xi_k^{\prime}{\underline{U}}_{k-1}^{-1}\xi_k > 0$ which implies 
\begin{equation*}
{\underline{U}}_k = \begin{pmatrix}
U(0) & \xi_k^{\prime} \\
\xi_k & {\underline{U}}_{k-1}
\end{pmatrix} ,
\end{equation*}
is positive definite.
\end{proof}

\begin{proof} [Proof of Theorem~\ref{thm:steinStab}] 
If $A(z)$ is Schur-Stable then we can construct a causal VAR($k$) process $\{X_t\}$ with innovation variance as the identity matrix $1_m$ and $A$ as the coefficient matrix. Then ${\underline{U}}_{k-1}) = \Gamma_{k-1}$ is the required positive definite block Toeplitz matrix where $\Gamma_k $ is the variance  $Var\{ X_t, X_{t-1}, \ldots, X_{t-k}\}.$

For the converse suppose   ${\underline{U}}_{k-1} \in {\mathfrak{T}}^{m,k-1}_{++}$  exists  with ${\tilde{S}}_k(A, {\underline{U}}_{k-1}) \in {\mathscr{S}}^{mk}_{+}$  and ${{S}}_k(A, {\underline{U}}_{k-1}) \in {\mathscr{S}}^{m}_{++}. $ Because of the block Toeplitz structure, 
\[ {\tilde{S}}_k(A, {\underline{U}}) = \begin{pmatrix} {S}_k(A, {\underline{U}}) & \star \\ \star & 0 \end{pmatrix}.\]
Positive definiteness of ${\tilde{S}}_k(A, {\underline{U}}_{k-1})$ would then imply
\[ {\tilde{S}}_k(A, {\underline{U}}) = \begin{pmatrix} {S}_k(A, {\underline{U}}) & 0  \\ 0 & 0 \end{pmatrix}.\]
 Let $w^* = (w_1^{*},\cdots, w_p^{*})^{\prime}$ be any left  eigenvector of $\tilde{A}$ corresponding to an eigenvalue $\lambda$. Because ${{S}}_k(A, {\underline{U}}_{k-1}) \in {\mathscr{S}}^{m}_{++},$ following the arguments given in the proof of Theorem~\ref{thm:stabCond}, we have $|\lambda| < 1$. 
\end{proof}

\end{document}